\documentclass{amsart}
\usepackage{amsmath}
\usepackage{amssymb}
\usepackage{amsthm}
\usepackage{enumerate}
\usepackage[dvips]{graphicx}
\usepackage{pb-diagram}
\theoremstyle{definition}
\newtheorem{definition}{Definition}[section]
\theoremstyle{plain}
\newtheorem{lem}[definition]{Lemma}
\newtheorem{thm}[definition]{Theorem}
\newtheorem{prop}[definition]{Proposition}
\newtheorem{cor}[definition]{Corollary}
\theoremstyle{remark}
\newtheorem{rem}[definition]{Remark}
\newtheorem{notation}[definition]{Notation}

\newtheorem{mainthm}[definition]{Main Theorem}

\newcommand{\cdfr}{C_{\text{df}}^r}
\newcommand{\cdfz}{C_{\text{df}}^0}
\newcommand{\vdfr}{\operatorname{\mathbf{VB}}_{\text{df}}^r}
\newcommand{\proj}{\operatorname{\mathbf{Proj}}}
\newcommand{\projsec}{\operatorname{Proj}}
\newcommand{\funcvec}{\Gamma}
\newcommand{\bdfr}{\operatorname{\mathbf{BS}}_{\text{df}}^r}
\newcommand{\bilsp}{\operatorname{\mathbf{BS}}}
\newcommand{\funcbil}{\Gamma_b}
\newcommand{\rspec}{\operatorname{Spec}_r}
\newcommand{\nash}{\mathcal N}
\newcommand{\gl}{\operatorname{GL}}
\newcommand{\sgl}{\operatorname{SGL}}
\newcommand{\grassman}{\operatorname{\mathbf{Gr}}}

\makeatletter
\@namedef{subjclassname@2010}{%
  \textup{2010} Mathematics Subject Classification}
\makeatother

\begin{document}

\title[definable vector bundles and bilinear spaces]{Definable $C^r$ vector bundles and bilinear spaces in an o-minimal structure and their homotopy theorems}
\author[M. Fujita]{Masato Fujita}
\address{Department of Liberal Arts,
Japan Coast Guard Academy,
5-1 Wakaba-cho, Kure, Hiroshima 737-8512, Japan}
\email{fujita.masato.p34@kyoto-u.jp}

\begin{abstract}
Consider an o-minimal structure on the real field.
Let $M$ be a definable $C^r$ manifold, where $r$ is a nonnegative integer. 
We first demonstrate an equivalence of the category of definable $C^r$ vector bundles over $M$ with the category of finitely generated projective modules over the ring $\cdfr(M)$.
Here, the notation $\cdfr(M)$ denotes the ring of definable $C^r$ functions on $M$.
We also show an equivalence of the category of definable $C^r$ bilinear spaces over $M$ with the category of bilinear spaces over the ring $\cdfr(M)$.
The main theorems of this paper are homotopy theorems for definable $C^r$ vector bundles and definable $C^r$ bilinear spaces over $M$.
As an application, we show that the Grothendieck rings $K_0(\cdfr(M))$, $K_0(\cdfz(M))$  and the Witt ring $W(\cdfr(M))$ are all isomorphic.
\end{abstract}

\subjclass[2010]{Primary 03C64; Secondary 57R22, 19A49}

\keywords{o-minimal structure, vector bundle, Grothendieck ring, Witt ring}

\maketitle

\section{Introduction}\label{sec:intro}
O-minimal structure was initially a subject of mathematical logic, but it was realized that an o-minimal structure provides an excellent framework of geometry \cite{vdD}.
It is considered to be a generalization of semialgebraic geometry \cite{BCR}, and geometric assertions on semialgebraic sets are generalized to the o-minimal case, such as triangulation and trivialization \cite{vdD}. 

In real algebraic geometry, a systematic study of algebraic vector bundles over real algebraic varieties is given in a series of papers: \cite{B1,B2,B3,B4,B5,B6,B7,B8,B9,B10,B11}, and a part of the results is summarized in \cite[Chapter 12]{BCR}. 
Semialgebraic vector bundles are also studied in \cite{BCR,C}.
For instance, there is an equivalence of the category of finitely generated project modules over the ring of regular functions on an affine real algebraic variety with algebraic vector bundles over the variety by \cite[Proposition 12.1.12]{BCR}.
We also find a similar equivalence for semialgebraic objects in \cite[Corollary 12.7.6]{BCR}.
It is natural to generalize these studies to the o-minimal case.
A study on definable fiber bundles was already done in \cite{K1}, but we cannot find an assertion on an equivalence of the category of definable $C^r$ vector bundles with the category of finitely generated projective modules over the ring of definable $C^r$ functions.

We fix an o-minimal structure on the real field throughout this paper.
The term `definable' means `definable in the o-minimal structure' in this paper.
Let $M$ be a definable $C^r$ manifold, where $r$ is a nonnegative integer. 
Note that all definable $C^r$ manifolds are affine by \cite[Theorem 1.1]{K0} and \cite[Theorem 1.3]{Fisher}.
We use this fact without explicitly stated in this paper.
The notation $\cdfr(M)$ denotes the ring of definable $C^r$ functions on $M$.
We investigate several categories in this paper.
The first category is $\vdfr(M)$;  
the objects are definable $C^r$ vector bundles over $M$, and the arrows are definable $C^r$ $M$-morphisms between definable $C^r$ vector bundles.
The second  category is $\proj(R)$, where $R$ is a commutative ring.
The objects are finitely generated projective $R$-modules, and the arrows are homomorphisms between $R$-modules.

For any definable $C^r$ vector bundle $\xi$ over a definable $C^r$ manifold $M$, $\funcvec(\xi)$ denotes the set of all definable $C^r$ sections of $\xi$. 
It is a finitely generated projective $\cdfr(M)$-module.
The notation $\funcvec(\varphi)$ is the induced homomorphism from $\funcvec(\xi)$ to $\funcvec(\xi')$ for any definable $C^r$ $M$-morphism $\varphi:\xi \rightarrow \xi'$.  
$\funcvec$ is a covariant functor from $\vdfr(M)$ to $\proj(\cdfr(M))$.
The following is the first main theorem in this paper:

\begin{mainthm}\label{thm:main1}
Let $M$ be a definable $C^r$ manifold, where $r$ is a nonnegative integer. 
Then, the functor $\funcvec$ is an equivalence of the category $\vdfr(M)$ with the category $\proj(\cdfr(M))$.
\end{mainthm}

We define a definable $C^r$ bilinear space over a definable $C^r$ manifold in Section \ref{sec:equi_bil}.
A bilinear space over a commutative ring is defined in \cite{BCR, MH}.
We show an equivalence of bilinear spaces.
Let $M$ be a definable $C^r$ manifold, where $r$ is a nonnegative integer. 
The notation $\bdfr(M)$ is the category whose objects are definable $C^r$ bilinear spaces over $M$ and whose arrows are definable $C^r$ $M$-morphisms between definable $C^r$ bilinear spaces.
We also consider the category $\bilsp(R)$, where $R$ is a commutative ring.
The objects are bilinear spaces over the ring $R$ and the arrows are morphisms of bilinear spaces over $R$.
We also show that these categories are equivalent.

\begin{mainthm}\label{thm:main2}
Let $M$ be a definable $C^r$ manifold, where $r$ is a nonnegative integer. 
The category $\bdfr(M)$ is equivalent to the category $\bilsp(\cdfr(M))$.
\end{mainthm}

The homotopy theorem for vector bundles is well-known and, for instance, it is found in \cite{H}.
We extend the theorem to the general o-minimal and $C^r$ case.
\begin{mainthm}[Homotopy theorem for definable $C^r$ vector bundles]\label{thm:main30}
Consider a definable $C^r$ manifold $M$, where $r$ is a nonnegative integer. 
Let $U$ be a definable open subset of $M \times \mathbb R$ containing $M \times [0,1]$.
Let $\Xi$ be a definable $C^r$ vector bundle over $U$.
Then, two definable $C^r$ vector bundles $\Xi|_{M \times \{0\}}$ and $\Xi|_{M \times \{1\}}$ are definably $C^r$ isomorphic.
Here, the notation $\Xi|_{M \times \{0\}}$ denotes the restriction of the vector bundle $\Xi$ to $M \times \{0\}$.
\end{mainthm}

We can also prove a homotopy theorem for bilinear spaces. 
It is a generalization of the homotopy theorem for semialgebraic $C^0$ bilinear spaces in \cite[Corollary 15.1.9]{BCR}.
\begin{mainthm}[Homotopy theorem for definable $C^r$ bilinear spaces]\label{thm:main3}
Consider a definable $C^r$ manifold $M$, where $r$ is a nonnegative integer. 
Let $U$ be a definable open subset of $M \times \mathbb R$ containing $M \times [0,1]$.
Let $(\Xi,B)$ be a definable $C^r$ bilinear space over $U$.
Then, two definable $C^r$ bilinear spaces $(\Xi,B)|_{M \times \{0\}}$ and $(\Xi,B)|_{M \times \{1\}}$ are definably $C^r$ isometric.
Here, the notation $(\Xi,B)|_{M \times \{0\}}$ denotes the restriction of the bilinear space $(\Xi,B)$ to $M \times \{0\}$.
\end{mainthm}

We show isomorphisms between Grothendieck rings and Witt rings as an application of the above theorems.
The Grothendieck group of the ring is studied in K-theory \cite{Milnor,Weibel}.
It is defined as an abelian group generated by finitely generated projective modules modulo some equivalence relation.
The multiplication in the abelian group is defined as the tensor product of two representative projective modules.
The Witt ring \cite{MH} of the ring of semialgebraic functions is isomorphic to its Grothendieck ring by \cite[Theorem 15.1.2]{BCR}.
A similar assertion holds true for the ring of definable $C^r$ functions.
\begin{mainthm}\label{thm:main4}
Let $M$ be a definable $C^r$ manifold, where $r$ is a nonnegative integer. 
The Grothendieck ring $K_0(\cdfr(M))$ of the ring $\cdfr(M)$ is isomorphic to the Witt ring $W(\cdfr(M))$ of the same ring.
The Grothendieck rings $K_0(\cdfr(M))$ and $K_0(\cdfz(M))$ are also isomorphic.
\end{mainthm}

This paper is organized as follows:
In Section \ref{sec:preparation}, we first introduce several basic lemmas used through the remaining sections.
We show the equivalence of the category $\vdfr(M)$ with the category $\proj(\cdfr(M))$ in Section \ref{sec:equi_vec}.
An approximation theorem for sections of definable $C^r$ vector bundles and its applications are introduced in Section \ref{sec:appro}.
We demonstrate that the category $\bdfr(M)$ is equivalent to the category $\bilsp(C_{\text{df}}^r(M))$  in Section \ref{sec:equi_bil}.
The homotopy theorems for definable $C^r$ vector bundles and bilinear spaces are proved in Section \ref{sec:homotopy}.
As an application of these results, we demonstrate that the Grothendieck ring of $\cdfr(M)$ and the Witt ring of $\cdfr(M)$ are isomorphic in Section \ref{sec:grothendieck}.

We summarize basic notations used in this paper.
The notation $\overline{V}$ denotes the closure of a subset $V$ of a topological space.
The complement of the set $V$ is denoted by $V^c$ in this paper.
The transpose of a matrix $X$ is denoted  by ${}^t\!X$.
The notation $[0,1]$ denotes the closed interval $\{x \in \mathbb R\;|\; 0 \leq x \leq 1\}$ in $\mathbb R$.

\section{Preliminary}\label{sec:preparation}
We provide several basic lemmas in this section.
We first review the definition of definable $C^r$ manifolds.
\begin{definition}[Definable manifolds] 
Let $r$ be a nonnegative integer.
A \textit{definable $C^r$ manifold} $M$ of dimension $n$ is a $C^r$ manifold having finite charts $\{(U_i, \phi_i)\}_{i=1}^q$ such that, for any $1 \leq i,j \leq q$,
\begin{itemize}
\item $U_i$ is an open subset of $M$, 
\item $\phi_i: U_i \rightarrow V_i$ is a homeomorphism, where $V_i$ is an definable open subset of $\mathbb R^n$, and, 
\item $\phi_i|_{U_i \cap U_j} \circ (\phi_j|_{\phi_j(U_i \cap U_j)})^{-1}: \phi_j(U_i \cap U_j) \rightarrow \phi_i(U_i \cap U_j)$ is a definable $C^r$ diffeomorphism if $U_i \cap U_j \not= \emptyset$.
\end{itemize}
Charts satisfying the above conditions are called \textit{definable $C^r$ charts} of $M$.
A $C^r$ map $f:M \rightarrow M'$ between definable $C^r$ manifolds is a \textit{definable $C^r$ map} if, for any definable $C^r$ charts $\{(U_i, \phi_i:U_i \rightarrow V_i)\}_{i=1}^q$ of $M$ and definable $C^r$ charts $\{(U'_i, \phi'_i:U'_i \rightarrow V'_i)\}_{i=1}^{q'}$ of $M'$, the map $(\phi'_j)^{-1}\circ f|_{U_i \cap f^{-1}(U'_j)}\circ \phi_i: \phi_i^{-1}(U_i \cap f^{-1}(U'_j)) \rightarrow V'_j$ is a definable $C^r$ map for any $1 \leq i \leq q$ and $1 \leq j \leq q'$.
An \textit{affine definable $C^r$ manifold} $M$ of dimension $n$ is a definable subset of a Euclidean space $\mathbb R^m$ such that there exist a finite definable open covering $\{U_i\}_{i=1}^q$ of $M$ in $\mathbb R^m$ and a family of definable $C^r$ diffeomorphisms onto definable open sets $\{\phi_i:U_i \rightarrow \Omega_i \subset \mathbb R^m\}_{i=1}^{q}$ with $M \cap U_i = \phi_i^{-1}((\mathbb R^n \times \{0\}) \cap \Omega_i)$ for any $1 \leq i \leq q$. 
We call the collection of pairs $(U_i,\phi_i)_{i=1}^q$ \textit{charts} of the affine definable $C^r$ manifold.
If a definable $C^r$ manifold $M$ has a definable $C^r$ immersion $\iota:M \rightarrow \mathbb R^m$, we call it an affine definable $C^r$ manifold identifying it with its image.
Note that all definable $C^r$ manifolds are affine by \cite[Theorem 1.1]{K0} and \cite[Theorem 1.3]{Fisher}.
We use this fact without explicitly stated in this paper.
\end{definition}

\begin{lem}\label{lem:zeroset}
Let $M$ be a definable $C^r$ manifold with $0 \leq r < \infty$. 
Given a definable closed subset $X$ of $M$,
there exists a definable $C^r$ function $f:M \rightarrow \mathbb R$ whose zero set is $X$.
\end{lem}
\begin{proof}
We assume that $M$ is a definable $C^r$ submanifold of a Euclidean space $\mathbb R^m$. 
Consider the closure $\overline{X}$ of $X$ in $\mathbb R^m$.
There exists a definable $C^r$ function $F:\mathbb R^m \rightarrow \mathbb R$ with $F^{-1}(0)=\overline{X}$ by \cite[Theorem C.11]{vdDM}.
The restriction of $F$ to $M$ satisfies the requirement.
\end{proof}

\begin{lem}\label{lem:sep}
Let $M$ be a definable $C^r$ manifold with $0 \leq r < \infty$. 
Let $X$ and $Y$ be closed definable subsets of $M$ with $X \cap Y = \emptyset$.
Then, there exists a definable $C^r$ function $f:M \rightarrow [0,1]$ with $f^{-1}(0)=X$ and $f^{-1}(1)=Y$.
\end{lem}
\begin{proof}
There exist definable $C^r$ functions $g,h:M \rightarrow \mathbb R$ with $g^{-1}(0)=X$ and $h^{-1}(0)=Y$ by Lemma \ref{lem:zeroset}. The function $f:M \rightarrow [0,1]$ defined by $f(x)=\frac{g(x)^2}{g(x)^2+h(x)^2}$ satisfies the requirement.
\end{proof}

\begin{lem}\label{lem:middle}
Let $M$ be a definable $C^r$ manifold with $0 \leq r < \infty$. 
Let $C$ and $U$ be closed and open definable subsets of $M$, respectively. 
Assume that $C$ is contained in $U$.
Then, there exists an open definable subset $V$ of $M$ with $C \subset V \subset \overline{V} \subset U$.
\end{lem}
\begin{proof}
There is a definable continuous function $h:M \rightarrow [0,1]$ with $h^{-1}(0)=C$ and $h^{-1}(1) = M \setminus U$ by Lemma \ref{lem:sep}.
The set $V=\{x \in M\,;\, h(x)<\frac{1}{2}\}$ satisfies the requirement.
\end{proof}

\begin{lem}[Fine definable open covering]\label{lem:covering}
Let $M$ be a definable $C^r$ manifold with $0 \leq r < \infty$. 
Let $\{U_i\}_{i=1}^q$ be a finite definable open covering of $M$.
For each $1 \leq i \leq q$, 
there exists a definable open subset $V_i$ of $M$ satisfying the following conditions:
\begin{itemize}
\item  the closure $\overline{V_i}$ in $M$ is contained in $U_i$ for each $1 \leq i \leq q$, and 
\item  the collection $\{V_i\}_{i=1}^q$ is again a finite definable open covering of $M$. 
\end{itemize}
\end{lem}
\begin{proof}
We inductively construct $V_i$ so that $\overline{V_i} \subset U_i$ and $\{V_i\}_{i=1}^{k-1} \cup \{U_i\}_{i=k}^q$ is a finite definable open covering of $M$.
We fix a positive integer $k$ with $k \leq q$.
Set $C_k = M \setminus (\bigcup_{i=1}^{k-1} V_i \cup \bigcup_{i=k+1}^{q} U_i)$.
The set $C_k$ is a definable closed subset of $M$ contained in $U_k$.
There exists a definable open subset $V_k$ of $M$ with $C_k \subset V_k \subset \overline{V_k} \subset U_k$ by Lemma \ref{lem:middle}.
It is obvious that $\{V_i\}_{i=1}^{k} \cup \{U_i\}_{i=k+1}^q$ is a finite definable open covering of $M$.
\end{proof}

\begin{lem}[Partition of unity]\label{lem:unity}
Let $M $ be a definable $C^r$ manifold.
Given a finite open definable covering $\{U_i\}_{i=1}^q$ of $M$, 
there exist nonnegative definable $C^r$ functions $\lambda_i$ on $M$ for all $1 \leq i \leq q$ such that $\sum_{i=1}^q\lambda_i = 1$ and the closure of the set $\{x \in M\;|\; \lambda_i(x)>0\}$ is contained in $U_i$. 
\end{lem}
\begin{proof}
Let $\{V_i\}_{i=1}^q$ be a finite definable open covering given in Lemma \ref{lem:covering}.   
There exists a $C^r$ definable function $f_i$ on $M$ with $f_i^{-1}(0)=M \setminus V_i$ by Lemma \ref{lem:zeroset}.
Set $\lambda_i=f_i^2/\sum_{j =1}^q f_j^2$.
The definable $C^r$ functions $\lambda_i$ on $M$ satisfy the requirements.
\end{proof}

The semialgebraic counterpart to the following lemma is found in 
\begin{lem}\label{lem:cover01}
Let $X$ be a definable subset of $\mathbb R^n$ and $\{V_j\}_{j=1}^p$ be a finite definable open covering of $X \times [0,1]$.
Then, there exist a finite definable open covering $\{U_i\}_{i=1}^q$ of $X$ and finite definable $C^r$ functions $0= \varphi_{i,0} < \cdots < \varphi_{i,k} < \cdots < \varphi_{i,r_i}=1$ on $U_i$ such that, for any $1 \leq i \leq q$ and $1 \leq k \leq r_i$, the definable set
\begin{equation*}
\{(x,t) \in U_i \times [0,1] \;|\; \varphi_{i,k-1}(x) \leq t \leq \varphi_{i,k}(x)\}
\end{equation*}
is contained in $V_j$ for some $1 \leq j \leq p$.
\end{lem}
\begin{proof}
Let $\pi:\mathbb R^n \times \mathbb R \rightarrow \mathbb R^n$ be the projection onto the first $n$ coordinates.
Apply the $C^r$ cell decomposition theorem \cite[Chapter 7, Theorem 3.2 and Chapter 7, Exercise 3.3]{vdD} to $\mathbb R^{n} \times \mathbb R$, then we have a definable $C^r$ cell decomposition $\{C_i\}_{i=1}^s$ partitioning $X \times [0,1]$ and $V_1 ,\ldots, V_p$.
Let $\{D_i\}_{i=1}^q$ be the subfamily of cells in $\{\pi(C_i)\}_{i=1}^s$ contained in $X$.
There exist definable $C^r$ functions $\psi_{i,0} = - \infty < \psi_{i,1} < \cdots < \psi_{i,l_i} < \psi_{i,l_i+1}=\infty$ on $D_i$ such that the cells contained in $X \times \mathbb R$ are one of the following forms:
\begin{align*}
&\{(x,t) \in D_i \times \mathbb R\;|\; t = \psi_{i,k}(x)\}\text{,}\\
&\{(x,t) \in D_i \times \mathbb R\;|\;  \psi_{i,k-1}(x) < t < \psi_{i,k}(x)\}\text{.}
\end{align*}
Since the cell decomposition partitions $X \times [0,1]$, there exist integers $k$ and $k'$ with $\psi_{i,k}=0$ and $\psi_{i,k'}=1$.
Taking a subsequence, we may assume that $0=\psi_{i,0} < \psi_{i,1} < \cdots < \psi_{i,l_i}=1$ on $D_i$.
Set $r_i=2l_i$ and define definable $C^r$ functions $\Psi_{i,k}$ on $D_i$ by
\begin{equation*}
\Psi_{i,k}(x)=\left\{\begin{array}{ll}
\psi_{i,k/2}(x) & \text{ if } k \text{ is even, }\\
\frac{\psi_{i,(k-1)/2}(x)+\psi_{i,(k+1)/2}(x)}{2} & \text{ otherwise.}
\end{array}\right.
\end{equation*} 
 
 We first show the following claim:
  \medskip
  
 {\textbf{Claim.}}
 For any $1 \leq i \leq q$ and $1 \leq k \leq r_i$, there exists a positive integer $j(i,k)$ such that the definable set
 \begin{equation*}
 \{(x,t) \in D_i \times [0,1] \;|\; \Psi_{i,k-1}(x) \leq t \leq \Psi_{i,k}(x)\} 
 \end{equation*}
is contained in $V_{j(i,k)}$.
\medskip

We demonstrate the above claim.
One of $k-1$ and $k$ is an even number. 
We assume that $k-1$ is even.
We can show the claim in the same way in the other case.
Since $\{V_j \}_{j=1}^p$ is an open covering and the decomposition partitions $V_1, \ldots V_p$, there exists $1 \leq j(i,k) \leq p$ such that the set $\{(x,t) \in D_i \times [0,1] \;|\; t = \Psi_{i,k-1}(x)=\psi_{i,(k-1)/2}(x)\}$ is contained in $V_{j(i,k)}$.
The set $\{(x,t) \in D_i \times \mathbb R\;|\;  \psi_{i,(k-1)/2}(x) < t < \psi_{i,(k+1)/2}(x)\}$ is contained in $V_{j(i,k)}$ because $V_{j(i,k)}$ is open and the decomposition partitions $V_{j(i,k)}$.
Hence, the definable set $\{(x,t) \in D_i \times [0,1] \;|\; \Psi_{i,k-1}(x) \leq t \leq \Psi_{i,k}(x)\}$ is contained in $V_{j(i,k)}$.
We have demonstrated the claim.
\medskip

Let $\pi_l:\mathbb R^n \rightarrow \mathbb R^l$ be the projection onto the first $l$ coordinates.
We inductively define definable open subsets $W_{i,l}$ of $\mathbb R^l$ and definable $C^r$ maps $\eta_{i,l}: W_{i,l} \rightarrow \pi_l(D_i)$ as follows:
When $l=1$, the definable set $\pi_1(D_i)$ is a single point set $\{a\}$ or an open interval $I$.
Set $W_{i,1}=\mathbb R$ and $\eta_{i,1}(x)=a$ if $\pi_1(D_i)$ is a single point set $\{a\}$.
Set $W_{i,1}=I$ and $\eta_{i,1}(x)=x$ if $\pi_1(D_i)$ is an open interval $I$.
When $l>1$, the definable set $\pi_l(D_i)$ is one of the following forms:
 \begin{align*}
&\{(x,t) \in \pi_{l-1}(D_i) \times \mathbb R\;|\; t = f(x)\}\text{,}\\
&\{(x,t) \in \pi_{l-1}(D_i) \times \mathbb R\;|\;  f_1(x) < t < f_2(x)\}\text{.}
\end{align*}
Here, $f,f_1$ and $f_2$ are definable $C^r$ functions on $\pi_{l-1}(D_i)$.
Set $W_{i,l}= W_{i,l-1} \times \mathbb R$ and $\eta_{i,l}(x,t)=(\eta_{i,l-1}(x),f(\eta_{i,l-1}(x)))$ in the former case.
Set $W_{i,l}=\{(x,t) \in W_{i,l-1} \times \mathbb R\;|\; f_1(\eta_{i,l-1}(x)) < t < f_2(\eta_{i,l-1}(x))\}$ and $\eta_{i,l}(x,t)=(\eta_{i,l-1}(x),t)$ in the latter case.

Set $W_i = W_{i,n}$ and $\eta_i = \eta_{i,n}$.  
It is obvious that $D_i$ is contained in $W_i$ and the restriction of $\eta_i$ to $D_i$ is the identity map.
For any $1 \leq i \leq q$ and $1 \leq k \leq r_i$, we define a definable set $X_{i,k}$ as follows: 
\begin{equation*}
X_{i,k}=\{x \in W_i\;|\; (x,t) \in V_{j(i,k)} \text{ for all } t \in \mathbb R \text{ with } \Psi_{i,k-1}(\eta_i(x)) \leq t \leq \Psi_{i,k}(\eta_i(x))\}\text{.}
\end{equation*}
It is obvious that $D_i$ is contained in $X_{i,k}$ by the above claim.
We show that $X_{i,k}$ is an open set.
Let $x \in X_{i,k}$ be fixed.
Consider the closed definable subset 
\begin{equation*}
Y_{i,k}(x)=\{t \in \mathbb R\;|\; \Psi_{i,k-1}(\eta_i(x)) \leq t \leq  \Psi_{i,k}(\eta_i(x))\}
\end{equation*}
of $\mathbb R$.
We also set 
\begin{equation*}
Z_{i,k}(x)=\{(x,t) \in X_{i,k} \times \mathbb R\;|\; \Psi_{i,k-1}(\eta_i(x)) \leq t \leq  \Psi_{i,k}(\eta_i(x))\}\text{.}
\end{equation*}
The definable continuous function $\rho$ on the closed definable set $Y_{i,k}(x)$ is defined as the distance between the point $(x,t)$ and the closed set $V_{j(i,k)}^c$.
Since $Y_{i,k}(x)$ is compact, the function $\rho$ takes the minimum $m$.
The minimum $m$ is positive because the intersection of $Z_{i,k}(x)$ with $V_{j(i,k)}^c$ is empty.
Take $y \in W_i$ sufficiently close to $x$ satisfying the following conditions:
\begin{itemize}
\item $\|y-x\|_n<\frac{m}{2}$, where $\| \cdot \|_n$ denotes the Euclidean norm in $\mathbb R^n$,
\item $\Psi_{i,k-1}(\eta_i(y)) < \Psi_{i,k}(\eta_i(x))$, 
\item $\Psi_{i,k}(\eta_i(y)) > \Psi_{i,k-1}(\eta_i(x))$, 
\item $|\Psi_{i,k-1}(\eta_i(y)) - \Psi_{i,k-1}(\eta_i(x))| < \frac{m}{2}$ and 
\item $|\Psi_{i,k}(\eta_i(y)) - \Psi_{i,k}(\eta_i(x))| < \frac{m}{2}$.
\end{itemize}
We lead a contradiction assuming that $y \not\in X_{i,k}$.
There exists a real number $t$ with $\Psi_{i,k-1}(\eta_i(y)) \leq t \leq \Psi_{i,k}(\eta_i(y))$ and $(y,t) \not\in V_{j(i,k)}$ by the assumption.
If $t \in Y_{i,k}(x)$, the distance between the point $(y,t)$ and $Z_{i,k}(x)$ is $\|y-x\|_n$ and less than $m$.
It is a contradiction to the assumption that $(y,t) \not\in V_{j(i,k)}$.
In the other case, we have 
\begin{equation*}
|t-\Psi_{i,k-1}(\eta_i(x))| \leq |\Psi_{i,k-1}(\eta_i(y)) - \Psi_{i,k-1}(\eta_i(x))| < \frac{m}{2}
\end{equation*}
or 
\begin{equation*}
|t-\Psi_{i,k}(\eta_i(x))| \leq |\Psi_{i,k}(\eta_i(y)) - \Psi_{i,k}(\eta_i(x))| < \frac{m}{2}\text{.}
\end{equation*}
The distance between the point $(y,t)$ and $Z_{i,k}(x)$ is less than $\sqrt{\left(\frac{m}{2}\right)^2+\left(\frac{m}{2}\right)^2}$.
It contradicts to the definition of $m$.
We have demonstrated that $X_{i,k}$ is open.

Set $U_i=\bigcap_{k=1}^{r_i} X_{i,k}$ and $\varphi_{i,k}=\Psi_{i,k} \circ \eta_i|_{U_i}$ for $1 \leq i \leq q$.
The set $U_i$ is a definable open set and $\varphi_{i,k}$ is a definable $C^r$ function on $U_i$.
The set $\{(x,t) \in U_i \times [0,1] \;|\; \varphi_{i,k-1}(x) \leq t \leq \varphi_{i,k}(x)\}$ is contained in $V_{j(i,k)}$ by the definition of $X_{i,k}$.
Since $X=\bigcup_{i=1}^qD_i$ and $D_i \subset U_i$, we have $X \subset \bigcup_{i=1}^q U_i$. 
\end{proof}

\section{Equivalence of definable $C^r$ vector bundles with projective modules over the ring of definable $C^r$ functions}\label{sec:equi_vec}

We define definable $C^r$ vector bundles.
The definition is lengthy, but it is almost the same as the definition of semialgebraic vector bundles given in \cite[Definition 12.7.1]{BCR}.

\begin{definition}[Definable $C^r$ vector bundle]
Let $\xi=(E,p,M)$ be an $\mathbb R$-vector bundle of rank $d$ over a  definable $C^r$ manifold $M$.
A family of local trivializations $(U_i,\varphi_i:U_i \times \mathbb R^d \rightarrow p^{-1}(U_i))_{i \in I}$ of the vector bundle $\xi$ is a \textit{definable $C^r$ atlas} of $\xi$  if $(U_i)_{i \in I}$ is a finite open definable covering of $M$ and, for all pairs $(i,j) \in I \times I$, the maps $\varphi_i^{-1} \circ \varphi_j|_{(U_i \cap U_j) \times \mathbb R^d}$ are definable $C^r$ maps.
Two definable $C^r$ atlases are equivalent if their union is still a definable $C^r$ atlas. 
A \textit{definable $C^r$ vector bundle} is a vector bundle equipped with an equivalence class of definable $C^r$ atlases.
By abuse of notation, we denote a definable $C^r$ vector bundle by $\xi=(E,p,M)$ specifying the total space, the base space and the map between them, but without specifying the definable $C^r$ atlas defining its structure.
The \textit{trivial bundle of rank $d$ over $M$} is the triple $(M \times \mathbb R^d, p, M)$, where $p$ is the natural projection from $M \times \mathbb R^d$ onto $M$.
The notation $\epsilon_M^d$ denotes the trivial bundle of rank $d$ over $M$.
A section $s$ of the vector bundle $\xi$ is a \textit{definable $C^r$ section} if the maps $\varphi_i^{-1} \circ s|_{U_i}:U_i \rightarrow U_i \times \mathbb R^d$ are definable $C^r$ maps for all $i \in I$.

Let $\xi=(E,p,M)$ and $\xi'=(E',p',M')$ be two definable $C^r$ vector bundles.
Let $(U_i,\phi_i)_{i \in I}$ and $(U'_j,\phi'_j)_{j \in J}$ be definable $C^r$ atlases of $\xi$ and $\xi'$, respectively.
A morphism between vector bundles $\psi=(u,f):\xi \rightarrow \xi'$ is a \textit{definable $C^r$ morphism} if $(\phi_j')^{-1}\circ u \circ\phi_i|_{(U_i \cap U_j') \times \mathbb R^d}$ are definable $C^r$ maps for all pairs $(i,j) \in I \times J$, and the map $f:M \rightarrow M'$ between base spaces are also definable and of class $C^r$.
Here, the map $u$ is a map between total spaces.
When two definable $C^r$ manifolds $M$ and $M'$ are identical and the map between base spaces is the identity map on $M$, we call it a \textit{definable $C^r$ $M$-morphism}.
In this case, we use the same symbol for the morphism and the map between total spaces by abuse of notation.
Two definable $C^r$ vector bundles $\xi$ and $\xi'$ over a  definable $C^r$ manifold $M$ are \textit{definably $C^r$ isomorphic} if there exist two definable $C^r$ $M$-morphism $\psi:\xi \rightarrow \xi'$ and $\psi':\xi' \rightarrow \xi$ such that the compositions of two morphisms $\psi \circ \psi'$ and $\psi' \circ \psi$ are the identity morphisms.
We call $\psi$ and $\psi'$ \textit{definable $C^r$ isomorphisms}.
\end{definition}

\begin{rem}\label{rem:vector_bundle}
The notation $\gl(d,\mathbb R)$ denotes the general linear group of degree $d$ with entries in $\mathbb R$. 
Let $\xi=(E,p,M)$ be a definable $C^r$ vector bundle of rank $d$ over a definable $C^r$ manifold $M$.
Let $(U_i,\phi_i:U_i \times \mathbb R^d \rightarrow p^{-1}(U_i))_{i=1}^q$ be a definable $C^r$ atlas of $\xi$.
The map $\phi_i|_{(U_i \cap U_j) \times \mathbb R^d} \circ \left(\phi_j|_{(U_i \cap U_j) \times \mathbb R^d}\right)^{-1}: (U_i \cap U_j) \times \mathbb R^d \rightarrow (U_i \cap U_j) \times \mathbb R^d$ induces a collection of definable $C^r$ maps $\{g_{ij}: U_i \cap U_j \rightarrow \gl(d,\mathbb R)\}$ satisfying $g_{ii} = \operatorname{id}$ and $g_{ij}(x)g_{jk}(x)=g_{ik}(x)$ for any $x \in U_i \cap U_j \cap U_k$.
The function $g_{ij}: U_i \cap U_j \rightarrow \gl(d,\mathbb R)$ is called a \textit{transition function} of $\xi$. 

On the other hand, we can construct a definable $C^r$ vector bundle of rank $d$ in the same way as the case of general topological vector bundles \cite[Section 4.2]{H} if a finite definable open covering $\{U_i\}_{i=1}^q$ and a family of definable $C^r$ map $\{g_{ij}: U_i \cap U_j \rightarrow \gl(d,\mathbb R)\}$ are given and satisfy the above conditions.
\end{rem}

The proof of the following proposition is straightforward and we omit it.
\begin{prop}\label{prop:hometc}
Let $r$ be a nonnegative integer.
Let $M$ and $N$ be definable $C^r$ manifolds.
\begin{enumerate}[(i)]
\item If $\xi$ is a definable $C^r$ vector bundle over $M$, and $f:N \rightarrow M$ is a definable $C^r$ map, then $f^*(\xi)$ is a definable $C^r$ vector bundle over $N$.
\item If $\xi$ and $\eta$ are definable $C^r$ vector bundles over $M$, the Whitney sum $\xi \oplus \eta$, the tensor product $\xi \otimes \eta$, the dual $\xi^{\vee}$ and $\operatorname{Hom}(\xi,\eta)$ are definable $C^r$ vector bundles over $M$.
\end{enumerate}
\end{prop}

The following two lemmas are important and used several times in this paper.
\begin{lem}\label{lem:sections}
Given a definable $C^r$ vector bundle $\xi=(E,p,M)$ of rank $d$, 
there exist a finite number of definable $C^r$ sections $s_1, \ldots, s_m$ of $\xi$ such that, for any $x \in M$, the vectors $s_1(x), \ldots, s_m(x)$ generate the fiber $\xi_x=p^{-1}(x)$. 
\end{lem}
\begin{proof}
Let $(U_i,\varphi_i:U_i \times \mathbb R^d \rightarrow p^{-1}(U_i))_{i=1}^q$ be a definable $C^r$ atlas of $\xi$ and $\{\lambda_i\}_{i =1}^q$ be a definable $C^r$ partition of unity subordinate to the covering $\{U_i\}_{i = 1}^q$ given in Lemma \ref{lem:unity}.
For any $1 \leq i \leq q$ and $1 \leq j \leq d$, define a definable $C^r$ section $s_{i,j}$ by $s_{i,j}(x)=\varphi_i(x,\lambda_i(x)e_j)$ for any $x \in U_i$ and $s_{i,j}(x)=\varphi_k(x,0)$ for any $x \in U_k \setminus U_i$ for all $k \not=i$, where $e_j$ is the $j$-th vector of the canonical basis of $\mathbb R^d$.
The sections $s_{i,j}$ obviously generate the vector space $p^{-1}(x)$ for any $x \in M$.
\end{proof}

\begin{lem}\label{lem:sections2}
Let $\xi=(E,p,M)$ be a definable $C^r$ vector bundle of rank $d$.
Consider a finite number of definable $C^r$ sections $s_1, \ldots, s_m$ of $\xi$ such that, for any $x \in M$, the vectors $s_1(x), \ldots, s_m(x)$ generate the fiber $\xi_x=p^{-1}(x)$. 
Then, for any definable $C^r$ section $s$ of $\xi$, there exist definable $C^r$ functions $c_1, \ldots, c_m$ on $M$ with $s = \sum_{j=1}^m c_j s_j$.
\end{lem}
\begin{proof}
Let $(U_i,\varphi_i:U_i \times \mathbb R^d \rightarrow p^{-1}(U_i))_{i=1}^q$ be a definable $C^r$ atlas of $\xi$.
Fix $1 \leq i \leq q$.
For any subset $J \subset \{1 ,\ldots, m \}$ of cardinality $d$, set 
\begin{equation*}
U_{i,J}=\{ x \in U_i\;|\; \text{the fiber } p^{-1}(x) \text{ is generated by } \{s_j(x) \;|\; j \in J\}\}\text{.}
\end{equation*}
The collection $\{U_{i,J}\}$ is a finite definable open covering of $M$.
The fiber $\xi_x = p^{-1}(x)$ is generated by $\{s_j(x)\}_{j \in J}$ for any $x \in U_{i,J}$.
Taking a finer finite definable open covering, we may assume that, 
$(U_i,\varphi_i:U_i \times \mathbb R^d \rightarrow p^{-1}(U_i))_{i=1}^q$ is a definable $C^r$ atlas and, for any $1 \leq i \leq q$, there exists a subset $J(i) \subset \{1 ,\ldots, m \}$ such that the fiber $\xi_x$ is generated by $\{s_j(x)\}_{j \in J(i)}$ for any $x \in U_i$.
Let $j_i(k)$ be the $k$-th smallest element of the subset $J(i) \subset \{1, \ldots, m\}$.
The map $\tau_i: \epsilon_{U_i}^d \rightarrow \xi|_{U_i}$ defined by $\tau_i(x,(a_1, \ldots, a_d))=\sum_{k=1}^d a_ks_{j_i(k)}(x)$ is a definable $C^r$ isomorphism.
Let $\Pi_i: U_i \times \mathbb R^d \rightarrow \mathbb R^d$ be the natural projection.
Let $\pi_k:\mathbb R^d \rightarrow \mathbb R$ be the projection onto the $k$-th coordinate for any $1 \leq k \leq d$. 
We have 
\begin{equation*}
s(x) = \sum_{k=1}^d \pi_k(\Pi_i(\tau_i^{-1}(s(x))))s_{j_i(k)}(x)
\end{equation*}
for any $x \in U_i$.
Define definable $C^r$ functions $b_{i,j}$ on $U_i$ as follows:
\begin{equation*}
b_{i,j}=\left\{\begin{array}{ll}
\pi_k(\Pi_i(\tau_i^{-1}(s(x)))) & \text{ if }j=j_i(k)\text{ for some }1 \leq k \leq d \text{,}\\
0 & \text{ otherwise.}
\end{array}\right.
\end{equation*}
For any $1 \leq i \leq q$, the collection of definable functions $\{b_{ij}\}_{j=1}^m$ satisfies the following equality: 
\begin{equation*}
s|_{U_i} = \sum_{j=1}^m b_{i,j}s_{j}|_{U_i}\text{.}
\end{equation*}

Let $\{\lambda_i\}_{i=1}^q$ be a definable $C^r$ partition of unity subordinate to the definable open covering $\{U_i\}_{i=1}^q$ given by Lemma \ref{lem:unity}.
For any $1 \leq i \leq q$ and $1  \leq k \leq m$,   the function $c_{i,k}: M \rightarrow \mathbb R$ is the definable $C^r$ function defined by 
\begin{equation*}
c_{i,k}(x) = \left\{\begin{array}{ll} \lambda_i(x)b_{i,k}(x) & \text{ if } x \in U_i \text{,}\\
0 & \text{ elsewhere. }
\end{array}\right.
\end{equation*}
The equality $s = \sum_{i=1}^q \left( \sum_{k=1}^m c_{i,k}s_{k}\right)$ follows from the following calculation:
\begin{align*}
\sum_{i=1}^q \left( \sum_{k=1}^m c_{i,k}(x)s_{k}(x)\right)
&=\sum_{1 \leq i \leq q, x \in U_i} \lambda_{i}(x) \left( \sum_{k=1}^m b_{i,k}(x)s_k(x)\right)\\ 
&= \sum_{1 \leq i \leq q, x \in U_i} \lambda_{i}(x)s(x) =\sum_{i=1}^q \lambda_{i}(x)s(x)\\
& = s(x)
\end{align*}
Set $c_j = \sum_{i=1}^q c_{i,j}$ for any $1 \leq j \leq m$.
We have $s = \sum_{j=1}^m c_j s_j$.
\end{proof}

The following lemma is similar to \cite[Theorem 12.1.7, Corollary 12.7.5]{BCR}. 
The proof is also similar.
\begin{lem}\label{lem:equiv}
Let $\xi=(E,p,M)$ be a definable $C^r$ vector bundle of rank $d$.
Then:
\begin{enumerate}[\upshape (i)]
\item There exists an injective definable $C^r$ morphism from $\xi$ into a trivial bundle $\epsilon_M^n$.
\item There exists another definable $C^r$ vector bundle $\xi'$ over $M$ such that the Whitney sum $\xi \oplus \xi'$ is definably $C^r$ isomorphic to a trivial bundle $\epsilon_M^n$.
\end{enumerate}
\end{lem}
\begin{proof}
We first show that the assertion (ii) follows from the assertion (i).
We first show that there exists a definable $C^r$ map $f:M \rightarrow \grassman(n,d)$ such that $\xi$ is definably $C^r$ isomorphic to $f^*(\gamma_{n,d})$. 
Here, $\grassman(n,d)$ denotes the Grassmanian of $d$-dimensional subspaces of an $n$-dimensional vector space, and $\gamma_{n,d}$ denotes the universal vector bundle over $\grassman(n,d)$.
Let $\psi: \xi \rightarrow \epsilon_M^n$ be an injective definable $C^r$ morphism.
Define $f: M \rightarrow \grassman(n,k)$ by $\{x\} \times f(x) = \psi(p^{-1}(x))$.
It suffices to prove that $f$ is a definable $C^r$ map.
Let $(U_i,\varphi_i:\epsilon_{U_i}^d \rightarrow \xi|_{U_i})_{i=1}^q$ be a definable $C^r$ atlas of $\xi$.
Define $\eta_{ij}:U_i \rightarrow \mathbb R^d$ by $(x,\eta_{ij}(x))=\psi(\varphi_i(x,e_j))$ for any $1 \leq i \leq q$ and $1 \leq j \leq d$, where $e_j$ is the $j$-th vector of the canonical basis of $\mathbb R^d$.
It is a definable $C^r$ map.
Since $f(x)$ is generated by $\eta_{ij}(x)$ for any $x \in U_i$, the map $f$ is also a definable $C^r$ map.

Let $\gamma_{n,d}^{\perp}$ denote the vector bundle orthogonal to $\gamma_{n,d}$.
We have $f^*(\gamma_{n,d}) \oplus f^*(\gamma_{n,k}^\perp) \simeq \epsilon_M$ because $\gamma_{n,d} \oplus \gamma_{n,d}^\perp \simeq \epsilon_{\grassman(n,d)}^n$.
The vector bundle $\xi'=f^*(\gamma_{n,d}^\perp)$ is a definable $C^r$ vector bundle satisfying the required condition.

We next show the assertion (i).
Let $s_1, \ldots, s_n$ be the sections of $\xi$ given in Lemma \ref{lem:sections}.
The definable $C^r$ morphism $\psi: \epsilon_M^n \rightarrow \xi$ defined by 
\begin{equation*}
\psi(x,a_1, \ldots, a_n)=\left(x,\sum_{i=1}^n a_is_i(x)\right)
\end{equation*}
is surjective.
Taking the dual, we obtain an injective definable $C^r$ morphism $\psi^{\vee}: \xi^{\vee} \hookrightarrow \left(\epsilon_M^n\right)^{\vee} \simeq  \epsilon_M^n$.
We have already shown that the condition (i) implies the condition (ii). 
Hence, there exists a definable $C^r$ vector bundle $\xi'$ over $M$ such that the Whitney sum $\xi^{\vee} \oplus \xi'$ is isomorphic to a trivial bundle $\epsilon_M^n$.
We have $\xi \oplus (\xi')^{\vee} \simeq \left(\xi^{\vee}\right)^{\vee} \oplus (\xi')^{\vee} \simeq \left(\xi^{\vee} \oplus \xi'\right)^{\vee} \simeq (\epsilon_M^n)^{\vee} \simeq \epsilon_M^n$.
In particular, there exists an injective definable $C^r$ morphism from $\xi$ into a trivial bundle $\epsilon_M^n$.
\end{proof}

We begin to prove Main Theorem  \ref{thm:main1}.
We first review the notations introduced in Section \ref{sec:intro}.
\begin{notation}
Let $M$ be a definable $C^r$ manifold, where $r$ is a nonnegative integer. 
The notation $\cdfr(M)$ denotes the ring of definable $C^r$ functions on $M$.
The objects of the category $\vdfr(M)$ are definable $C^r$ vector bundles over $M$ and its arrows are definable $C^r$ $M$-morphisms between them.
Given a commutative ring $R$, the category $\proj(R)$ is the category whose objects are finitely generated projective $R$-modules and whose arrows are homomorphisms between projective modules.

For any definable $C^r$ vector bundle $\xi$ over a definable $C^r$ manifold $M$, the notation $\funcvec(\xi)$ denotes the set of all definable $C^r$ sections of $\xi$. 
The notation $\funcvec(\varphi)$ is the induced homomorphism from $\funcvec(\xi)$ to $\funcvec(\xi')$ for any definable $C^r$ $M$-morphism $\varphi:\xi \rightarrow \xi'$ between definable $C^r$ vector bundles over $M$.  
\end{notation}

We next show that $\funcvec$ is a covariant functor from $\vdfr(M)$ to $\proj(\cdfr(M))$.

\begin{prop}\label{prop:func1}
Let $M$ be a definable $C^r$ manifold, where $r$ is a nonnegative integer. 
The map $\funcvec$ is a covariant functor from the category $\vdfr(M)$ to the category $\proj(\cdfr(M))$.
\end{prop}
\begin{proof}
The only non-trivial assertion is that the set of all definable $C^r$ sections $\funcvec(\xi)$ is a finitely generated projective $\cdfr(M)$-module for any definable $C^r$ vector bundle $\xi$.

We first show that $\funcvec(\xi)$ is a finitely generated $\cdfr(M)$-module.
There are definable $C^r$ sections $s_1, \ldots, s_n$ of $\xi$ which generate the fibers $\xi_x$ for all $x \in M$ by Lemma \ref{lem:sections}.
The $\cdfr(M)$-module $\funcvec(\xi)$ is generated by $s_1, \ldots, s_n$ by Lemma \ref{lem:sections2}.

We next show that $\funcvec(\xi)$ is a projective module.
We have a  definable $C^r$ vector bundle $\xi'$ such that $\xi \oplus \xi'$ is definably $C^r$ isomorphic to a trivial bundle $\epsilon_M^n$ by Lemma \ref{lem:equiv}.
We have $$\funcvec(\xi) \oplus \funcvec(\xi') \simeq \funcvec(\epsilon_M^n)=\left(\cdfr(M)\right)^n\text{.}$$
Since $\funcvec(\xi)$ is a direct summand of a free module of finite rank, it is a finitely generated projective module.
\end{proof}

We finally show Main Theorem  \ref{thm:main1}.
\begin{thm}\label{thm:main1_0}
Let $M$ be a definable $C^r$ manifold, where $r$ is a nonnegative integer. 
Then, the functor $\funcvec$ is an equivalence of the category $\vdfr(M)$ with the category $\proj(\cdfr(M))$.
\end{thm}
\begin{proof}
In order to show that $\funcvec$ is an equivalence of the categories, we have only to show that the functor $\funcvec$ is faithful and full, furthermore; for any finitely generated projective $\cdfr(M)$-module $P$, there is a definable $C^r$ vector bundle $\xi=(E,p,M)$ such that $\funcvec(\xi)$ is isomorphic to $P$ by \cite[Chapter IV, Section 4, Theorem 1]{Category}.

We first show that the functor $\funcvec$ is faithful.
Let $\xi=(E,p,M)$ and $\xi'=(E',p',M)$ be two definable $C^r$ vector bundles.
Let $u:\xi \rightarrow \xi'$ and $u':\xi \rightarrow \xi'$ be definable $C^r$ $M$-morphisms with $\funcvec(u)=\funcvec(u')$.
Fix an arbitrary element $v \in E$.
There exists a definable $C^r$ section $s$ with $s(p(v))=v$ by Lemma \ref{lem:sections}.
We have $u(v)=u(s(p(v)))=(\funcvec(u)(s))(p(v))=(\funcvec(u')(s))(p(v))=u'(s(p(v)))=u'(v)$.
We have shown that $u=u'$.

We next show that the functor $\funcvec$ is full.
Let $\xi=(E,p,M)$ and $\xi'=(E',p',M)$ be two definable $C^r$ vector bundles.
Let $\Phi: \funcvec(\xi) \rightarrow \Gamma(\xi')$ be a homomorphism.
We construct a definable $C^r$ $M$-morphism $u:\xi \rightarrow \xi'$ with $\funcvec(u)=\Phi$.
Fix an arbitrary element $v \in E$.
Set $x=p(v)$.
There exists a definable $C^r$ section $s$ with $s(x)=v$ by Lemma \ref{lem:sections}.
We define $u(v)$ as 
\begin{equation*}
u(v)=\Phi(s)(x) \text{.}
\end{equation*}
We first show that it is a well-defined map.
Let $\mathfrak{m}_x$ be the maximal ideal of $\cdfr(M)$ of definable $C^r$ functions on $M$ vanishing at the point $x$.
Let $s'$ be another section with $s'(x)=v$.
We have $s-s'\in \mathfrak{m}_x\funcvec(\xi)$.
We get $\Phi(s)-\Phi(s') \in \mathfrak{m}_x\funcvec(\xi')$ because $\Phi$ is a homomorphism.
It means that $\Phi(s)(x)=\Phi(s')(x)$, and $u$ is well-defined.
We have only to show that $u$ is a definable $C^r$ $M$-morphism because the equality $\funcvec(u)=\Phi$ is obviously satisfied.

Let $(U_i,\varphi_i:U_i \times \mathbb R^d \rightarrow p^{-1}(U_i))_{i=1}^q$ and $(U'_i,\varphi'_j:U'_j \times \mathbb R^{d'} \rightarrow (p')^{-1}(U'_j))_{j=1}^{q'}$ be definable $C^r$ atlases of $\xi$ and $\xi'$, respectively.
Let $s_1 ,\ldots s_m$ be sections of $\xi$ generating the fibers $\xi_x=p^{-1}(x)$ for all $x \in M$ given by Lemma \ref{lem:sections}.
Fix $1 \leq i \leq q$ and $1 \leq j \leq q'$ with $U_i \cap U'_j \not=\emptyset$.
We have only to show that the restriction of $u$ to $U_i \cap U_j'$ is a definable $C^r$ $U_i \cap U_j'$-morphism.
For any subset $J \subset \{1 ,\ldots, m \}$ of cardinality $d$, set 
\begin{equation*}
U_{i,J}=\{ x \in U_i\;|\; \text{the fiber } \xi_x \text{ is generated by } \{s_j \;|\; j \in J\}\}\text{.}
\end{equation*}
Taking a finer finite definable open covering and arranging the order of the sections, we may assume that the fiber $\xi_x$ is generated by $\{s_1(x), \ldots, s_d(x)\}$ without loss of generality.
Let $u':(U_i \cap U_j') \times \mathbb R^d \rightarrow (U_i \cap U_j') \times \mathbb R^{d'}$ be the definable $C^r$ $U_i \cap U_j'$-morpshism defined by $u'(x,v)=(\varphi'_j)^{-1} \circ u \circ \varphi_i(x,v)$.
Let $\pi_i:U_i \times \mathbb R^d \rightarrow \mathbb R^d$ be the natural projection.
Consider the map $\tau:(U_i \cap U_j') \times \mathbb R^d \rightarrow (U_i \cap U_j') \times \mathbb R^d$ given by 
\begin{equation*}
\tau(x,(a_1, \ldots,a_d)) = \left(x, \sum_{k=1}^d a_k \pi_i(\varphi_i^{-1}(s_k(x)))\right)
\end{equation*}
for any $x \in U_i \cap U_j'$ and $(a_1, \ldots,a_d) \in \mathbb R^d$.
It is obviously a definable $C^r$ $U_i \cap U_j'$-isomorphism.
Consider the following commutative diagram:
\begin{equation*}
\begin{diagram}
\node{(U_i \cap U_j') \times \mathbb R^d} \arrow{se,t}{u'\circ\tau} \arrow{s,r}{\tau}\\
\node{(U_i \cap U_j') \times \mathbb R^d} \arrow{e,t} {u'} \arrow{s,r} {\varphi_i|_{(U_i \cap U_j') \times \mathbb R^d}} \node{(U_i \cap U_j') \times \mathbb R^{d'}} \arrow{s,r}{\varphi'_j|_{(U_i \cap U_j') \times \mathbb R^{d'}}}\\
\node{p^{-1}(U_i \cap U_J')} \arrow{r,t}{u|_{(U_i \cap U_j') \times \mathbb R^d}} \node{(p')^{-1}(U_i \cap U_j')}
\end{diagram}
\end{equation*}
Here, the map $u'\circ \tau$ is a map given by 
\begin{equation*}
u'\circ \tau(x,(a_1, \ldots,a_d)) = \left(x, \sum_{k=1}^d a_k \pi_i((\varphi_j')^{-1}(\Phi(s_k)(x)))\right)\text{.}
\end{equation*}
It is obviously a definable $C^r$ map.
The restriction of $u$ to $U_i \cap U_j'$ is also a definable $C^r$ $U_i \cap U_j'$-morphism.
Hence, the map $u$ is a definable $C^r$ $M$-morphism.

We finally demonstrate that, for any finitely generated projective $\cdfr(M)$-module $P$, there exists a definable $C^r$ vector bundle $\xi=(E,p,M)$ such that $\funcvec(\xi)$ is isomorphic to $P$.
We may assume that $M$ is connected without loss of generality.
Since $P$ is a direct summand of a free module by \cite[Proposition A3.1]{Eisenbud}, we may assume that $P \subset (\cdfr(M))^n$.
Let $s_1, \ldots, s_m \in (\cdfr(M))^n$ be generators of $P$.
Set 
\begin{equation*}
E=\{(x,v) \in M \times \mathbb R^n\;|\; v = \sigma(x) \text{ for some } \sigma \in P\}\text{,}
\end{equation*}
and $p:E \rightarrow M$ be the natural projection.

We demonstrate that $\xi=(E,p.M)$ is a definable $C^r$ vector bundle.
Let $\mathfrak{m}_x$ denote the set of definable $C^r$ functions on $M$ vanishing at $x$ for any $x \in M$.
The $\cdfr(M)_{\mathfrak{m}_x}$-module $P_{\mathfrak{m}_x}$ is free by \cite[Exercise 4.11]{Eisenbud}.
For all subsets $I \subset \{1, \ldots, m\}$, we define the subsets $U_I$ of $M$ by 
\begin{equation*}
\{x \in M\;|\;  \{s_i\}_{i \in I} \text{ is a basis of } P_{\mathfrak{m}_x}\}\text{.}
\end{equation*}
The set $U_I$ is open.
In fact, if $x \in U_I$, there exist $a_j \in \cdfr(M) \setminus \mathfrak{m}_x$ and $b_{j,k} \in \cdfr(M)$ for any $j \not\in I$ and $k \in I$ with $a_j(x)s_j(x) = \sum_{k \in I}b_{j,k}(x)s_k(x)$.
Set $V=\{y \in M\;|\; a_j(y)\not=0 \text{ for any } j \not\in I\}$, then $V$ is an open neighbourhood of $x$.
The set $\{s_i\}_{i \in I}$ is a basis of the localization $P_{\mathfrak{m}_y}$ for any $y \in V$. 
Hence, $V$ is contained in $U_I$.
It means that $U_I$ is open.
Since $M$ is connected, the rank of $P_{\mathfrak{m}_x}$ is constant, say $d$.
Let $A_I$ be the $n \times d$-matrix whose column vectors are $\{s_i\}_{i \in I}$.
It is obvious that $x \in U_I$ if and only if at least one of $d \times d$-minor determinants of $A_I$ is not zero at $x$.
Hence, $U_I$ is a definable set.
Let $\mathcal P$ be the collection of all subsets of $\{1, \ldots, m\}$ with $U_I \not= \emptyset$.
It is obvious that the family of definable open sets $\{U_I\}_{I \in \mathcal P}$ covers $M$.
We have shown that $\{U_I\}_{I \in \mathcal P}$ is a finite definable open covering of $M$.

Fix an arbitrary subset  $I \in \mathcal P$.
Let $j_I(k)$ be the $k$-th smallest element of the subset $I \subset \{1, \ldots, m\}$.
Define $\varphi_I:U_I \times \mathbb R^d \rightarrow p^{-1}(U_I)$ by 
\begin{equation*}
\varphi_I(x,(a_1, \ldots, a_d))= \sum_{k=1}^d a_ks_{j_I(k)}(x)\text{.}
\end{equation*}
It is obviously definable and of class $C^r$, and for any $x \in U_I$, the restriction $\varphi_I|_{\{x \} \times \mathbb R^d}$ of $\varphi_I$ to $\{x \} \times \mathbb R^d$ is an isomorphism between vector spaces.
Let $J \subset \{1, \ldots, m\}$ be another subset with $U_J \not= \emptyset$ and $U_I \cap U_J \not= \emptyset$.
We want to show that $\left(\varphi_J|_{(U_I \cap U_J) \times \mathbb R^d} \right)^{-1} \circ \varphi_I|_{(U_I \cap U_J) \times \mathbb R^d}$ is a definable $C^r$ diffeomorphism.
It is definable because $\varphi_I$ and $\varphi_J$ are definable.
It is also obvious that it is a bijective. 
We have only to show that it is a $C^r$ diffeomorphism.
Let $x \in U_I \cap U_J$ be an arbitrary point.
Since both $\{s_k\;|\;k \in I\}$ and $\{s_k\;|\;k \in J\}$ are bases of $P_{\mathfrak{m}_x}$, there exists $g \in \gl\left(d,\cdfr(M)_{\mathfrak{m}_x}\right)$ with $\left(\varphi_J|_{(U_I \cap U_J) \times \mathbb R^d} \right)^{-1} \circ \varphi_I|_{(U_I \cap U_J) \times \mathbb R^d}(y,v)=(y,g(y)v)$ for any $y \in M$ sufficiently close to $x$.
It shows that $\left(\varphi_J|_{(U_I \cap U_J) \times \mathbb R^d} \right)^{-1} \circ \varphi_I|_{(U_I \cap U_J) \times \mathbb R^d}$ is a $C^r$ diffeomorphism.
We have shown that $(U_I,\varphi_I)_{I \in \mathcal P}$ is a definable $C^r$-atlas of $\xi$.

It remains to show that $P=\funcvec(\xi)$.
The inclusion $P \subset \funcvec(\xi)$ is obvious because $\funcvec(\xi)$ contains the generators $s_1, \ldots, s_m$ of $P$.
We demonstrate the opposite inclusion.
Let $s \in \funcvec(\xi)$ be fixed.
The fiber $p^{-1}(x)$ is generated by $s_1, \ldots, s_m$ by the definition of the total space $E$.
Hence, there exist definable $C^r$ function $c_1, \ldots, c_m$ on $M$ with $s = \sum_{i=1}^m c_is_i$ by Lemma \ref{lem:sections2}.
We have shown $s \in P$ because $s_1, \ldots, s_m$ are generators of $P$.
\end{proof}

\begin{prop}\label{funcvecbun_plus}
Let $M$ be a definable $C^r$ manifold, where $r$ is a nonnegative integer. 
Then, we have $\funcvec(\xi_1 \oplus \xi_2)= \funcvec(\xi_1) \oplus \funcvec(\xi_2)$ and $\funcvec(\xi_1 \otimes \xi_2)= \funcvec(\xi_1) \otimes \funcvec(\xi_2)$ for any definable $C^r$ vector bundles $\xi_1$ and $\xi_2$ over $M$.
\end{prop}
\begin{proof}
We omit the proof.
\end{proof}

\section{Approximation of sections of definable $C^r$ vector bundles}\label{sec:appro}
This short section is devoted for demonstrating that any definable continuous section of a definable $C^r$ vector bundle can be approximated by a definable $C^r$ section.
We introduce two lemmas prior to the proof of the above assertion.
\begin{lem}\label{lem:appro_sub}
Let $C$ be a definable closed subset of $\mathbb R^n$.
Let $U \subset C \times \mathbb R^q$ be a definable open neighbourhood of $C \times \{0\}$ contained in $C \times B_q(0;1)$.
Then, there exists a positive definable continuous function $\rho:C \rightarrow \mathbb R$ with  $\{(x,v) \in C \times \mathbb R^q\;|\; \|v\|_q < \rho(x) \} \subset U$.
Here, $B_q(0;r)$ denotes the open ball in $\mathbb R^q$ centred at the origin of radius $r$, and $\|v\|_q$ denotes the Euclidean norm of a vector $v$ in $\mathbb R^q$.
\end{lem} 
\begin{proof}
Since $U$ is an open neighbourhood of $C \times \{0\}$, for any $x \in C$, there exist positive real numbers $\delta_x$ and $s_x$ such that the set $\{y \in C\;|\; \|y-x\| < \delta_x\} \times B_q(0;s_x)$ is contained in $U$.
Set $S = \{ s \in \mathbb R\;|\; \exists x \in C \subset \mathbb R^n, \|x\|_n = s\}$.
For any $s \in S$, the notation $C_s$ denotes the definable closed set $\{x \in C\;|\; \|x\|_n = s\}$.
Since the definable closed set $C_s$ is compact,
there exists a positive real number $t$ such that $C_s \times B_q(0;t)$ is contained in $U$.
The function $\phi:S \rightarrow \mathbb R$ given by 
\begin{equation*}
\phi(s)=\sup \{t \in \mathbb R\;|\; C_s \times B_q(0;t) \subset U\}
\end{equation*}
is a well-defined definable function, and positive for any point in $S$.

There exist finite points $s_1, \cdots, s_p \in S$ such that the restriction of $\phi$ to $S \setminus \{ s_1, \cdots, s_p\}$ is continuous by \cite[Chapter 3, Theorem 1.2]{vdD}.
We may assume that $S \setminus \{ s_1, \cdots, s_p\}$ is open by enlarging the discrete set $\{ s_1, \cdots, s_p\}$ if necessary.
Let $T$ be a connected component of $S \setminus \{ s_1, \cdots, s_p\}$ such that $s_i$ is contained in the closure of $T$.
The notation $\operatorname{graph}(\phi|_T)$ denotes the graph of the restriction of $\phi$ to $T$.
We show that $(s_i ,0)$ is not contained in the closure $\overline{\operatorname{graph}(\phi|_T)}$.
Assume the contrary.
There exists a convergent sequence $\{t_m \} \subset T$ with $t_m \to s_i$ and $\phi(t_m) \to 0$ as $m \to \infty$.
It means that, for any positive real number $\varepsilon > 0$, there exist a positive integer $m_\varepsilon$ and $x_\varepsilon \in C_{t_{m_\varepsilon}}$ with $\{x_\varepsilon\} \times B(0;\varepsilon) \not\subset U$.
There exists a real number $R$ with $t_m \leq R$ for any positive integer $m$ because $\{t_m \}$ is a convergent sequence.
Set $D=\{x \in C\;|\; \|x\|_n \leq R\}$. 
We have $x_{\varepsilon} \in D$ for any $\varepsilon > 0$.
Let $\{\varepsilon_m\}$ be a sequence of positive real numbers converging to $0$.
Since $D$ is compact, we may assume that the sequence $\{x_{\varepsilon_m}\}$ converges to a point $x' \in C$ by taking a subsequence of $\{\varepsilon_m\}$ if necessary.
Since $U$ is an open neighbourhood of $C \times \{0\}$, there exist positive real numbers $\delta_{x'}$ and $s_{x'}$ such that the set $\{y \in C\;|\; \|y-x'\| < \delta_{x'}\} \times B_q(0;s_{x'})$ is contained $U$.
We have $\{x_{\varepsilon_m}\} \times B_q(0;s_{x'}) \subset U$ for sufficiently large $m$.
On the other hand, for sufficiently large $m$, we have $\varepsilon_m < s_{x'}$ and we get $\{x_{\varepsilon_m}\} \times B_q(0;\varepsilon_m) \not\subset U$ by the assumption.
It is a contradiction.
We have shown that $(s_i ,0) \not\in \overline{\operatorname{graph}(\phi|_T)}$.

There exist positive real numbers $y_1, \ldots, y_p$ and $\eta_1, \ldots, \eta_p$ such that  the restriction of $\phi$ to the intersection $S \cap ( s_i -\eta_i,s_i+\eta_i)$ is larger than $y_i$ for all $1 \leq i \leq p$ because  $(s_i ,0) \not\in \overline{\operatorname{graph}(\phi|_T)}$.
Set $V_0 = S \setminus \{ s_1, \cdots, s_p\}$ and $V_i = ( s_i -\eta_i,s_i+\eta_i)$ for $1 \leq i \leq p$.
The collection of definable open sets $\{V_i\}_{i=0}^p$ is a definable open covering of the neighborhood $V = \bigcup_{i=0}^p V_i$ of $S$. 
Let $\{\tau_i: V \rightarrow \mathbb R\}_{i=0}^p$ be a definable continuous partition of unity subordinate to the open covering $\{V_i\}_{i=0}^p$ given by Lemma \ref{lem:unity}.
Define a definable continuous function $g_0 : V_0 \rightarrow \mathbb R$ by $g_0(x)= \phi(x)/2$.
We also define a definable continuous function $g_i : V_i \rightarrow \mathbb R$ by $g_i(x)= y_i/2$ for any $i=1,\ldots,p$.
Then, the function $\rho:C \rightarrow \mathbb R$ defined by $\rho(x)=\sum_{i=0}^q \tau_i(x)g_i(x)$  is a well-defined definable continuous function with $0 < \rho(x) < \phi(x)$ for any $x \in C$.
The definable continuous function $\rho$ satisfies the desired properties.
\end{proof}

The following lemma is \cite[Theorem 1.1]{Escribano}.
\begin{lem}\label{lem:escribano_appro}
Let $M$ and $N$ be definable $C^r$ manifolds. 
Assume that $N$ is a definable $C^r$ submanifold of $\mathbb R^n$.  
Let $f:M \rightarrow N$ be a definable continuous map and $\varepsilon:M \rightarrow \mathbb R$ be a positive definable continuous function on $M$. 
There exists a definable $C^r$ function $g:M \rightarrow \mathbb R$ such that $\| f(x)-g(x)\|_n < \varepsilon(x)$ for any $x \in M$. 
\end{lem}

We show an approximation theorem for definable $C^r$ sections using the above lemmas.
\begin{thm}\label{thm:appro_section}
Let $\xi=(E,p,M)$ be a definable $C^r$ vector bundle of rank $d$ over a definable $C^r$ manifold $M$, where $r$ is a nonnegative integer.
Let $\sigma:M \rightarrow E$ be a definable continuous section of $\xi$.
For any definable open neighbourhood $U$ of $\sigma(M)$ in $E$, there exists a definable $C^r$ section $s: M \rightarrow E$ with $s(M) \subset U$.
\end{thm}
\begin{proof}
We may assume that $M$ is a definable $C^r$ submanifold of a Euclidean space $\mathbb R^n$ which is simultaneously closed in $\mathbb R^n$.
In fact, $M$ is a definable $C^r$ submaniold of a Euclidean space $\mathbb R^n$ because $M$ is affine.
Take a definable $C^r$ function $H$ on $\mathbb R^n$ with $H^{-1}(0)=\overline{M} \setminus M$ using \cite[Theorem C.11]{vdDM}.
The image of $M$ under the definable immersion $\iota:M \rightarrow \mathbb R^{n+1}$ given by $\iota(x)=\left(x,1/H(x)\right)$ is a definable closed subset of $\mathbb R^{n+1}$.

Let $s_1, \ldots, s_m$ be definable $C^r$ sections of $\xi$ given in Lemma \ref{lem:sections}.
There exist definable continuous functions $\alpha_1, \ldots, \alpha_m$ on $M$ with $\sigma(x) = \sum_{i=1}^m \alpha_i(x)s_i(x)$ for any $x \in M$ by Lemma \ref{lem:sections2}.
Define $\tau:M \times \mathbb R^m \rightarrow E$ by $\tau(x,(c_1,\ldots,c_m))=\sum_{i=1}^m(\alpha_i(x)+c_i)s_i(x)$.
It is a definable continuous map.
The definable open set $V = \tau^{-1}(U)$ is a definable open neighbourhood of $M \times \{0\}$. 
Taking an intersection of $V$ with $M \times B_m(0;1)$, we may assume that $V$ is contained in $M \times B_m(0;1)$.
There exists a positive definable continuous function $\rho$ on $M$ such that $\{(x,v) \in M \times \mathbb R^m\;|\; \|v\|_m < \rho(x)\} \subset V$ by Lemma \ref{lem:appro_sub}. 
Take a definable $C^r$ approximation $\beta_i$ of $\alpha_i$ with $|\beta_i(x) - \alpha_i(x)| < \rho(x)/\sqrt{m}$ for any $ 1 \leq i \leq m$ using Lemma \ref{lem:escribano_appro}, then the definable $C^r$ section $s: M \rightarrow E$ given by $s(x)=\sum_{i=1}^m \beta_i(x)s_i(x)$ satisfies the requirement.
\end{proof}

\begin{cor}\label{cor:bundle_iso}
Let $r$ be a nonnegative integer.
If two definable $C^r$ vector bundles over a definable $C^r$ manifold are definably $C^0$ isomorphic, they are definably $C^r$ isomorphic. 
\end{cor}
\begin{proof}
Let $\xi_1$ and $\xi_2$ be definable $C^r$ vector bundles over a definable $C^r$ manifold $M$ which are definably $C^0$ isomorphic.
Remember that $\operatorname{Hom}(\xi_1,\xi_2)$ is a definable $C^r$ vector bundle over $M$ by Proposition \ref{prop:hometc}. 
Set 
\begin{equation*}
\operatorname{Iso}(\xi_1,\xi_2) = \{ (\phi,x)\;|\; \phi \text{ is an isomorphism between } (\xi_1)_x \text{ and } (\xi_2)_x \}\text{.}
\end{equation*}
It is a definable open subset of the total space of the vector bundle $\operatorname{Hom}(\xi_1,\xi_2)$.
A definable $C^0$ isomorphism between $\xi_1$ and $\xi_2$ corresponds to a definable continuous section of  $\operatorname{Hom}(\xi_1,\xi_2)$ whose image is contained in $\operatorname{Iso}(\xi_1,\xi_2)$.
There exists a definable $C^r$ section of $\operatorname{Hom}(\xi_1,\xi_2)$ contained in $\operatorname{Iso}(\xi_1,\xi_2)$ by Theorem \ref{thm:appro_section}.
This section corresponds to a definable $C^r$ isomorphism between $\xi_1$ and $\xi_2$.
\end{proof}

We give another important corollary of Lemma \ref{lem:escribano_appro}.
\begin{thm}\label{thm:bundle_iso2}
Let $r$ be a nonnegative integer.
Any definable $C^0$ vector bundle over a definable $C^r$ manifold is definably $C^0$ isomorphic to a definable $C^r$ vector bundle.
\end{thm}
\begin{proof}
Let $\xi$ be a definable $C^0$ vector bundle of rank $d$ over a definable $C^r$ manifold $M$.
There exist a trivial bundle $\epsilon_M^n$ over $M$ and a definable $C^0$ vector subbundle $\xi^\perp$ of $\epsilon_M^n$ such that $\xi$ is a definable $C^r$ vector subbundle of $\epsilon_M^n$ and $\xi \oplus \xi^\perp = \epsilon_M^n$ by Lemma \ref{lem:equiv}.
Let $\varphi: M \rightarrow \grassman(n,d)$ be the definable $C^0$ map naturally induced from the vector bundle $\xi$.
Let $\Tilde{\varphi}: M \rightarrow \grassman(n,d)$ be a definable $C^r$ approximation of $\varphi$ given by Lemma \ref{lem:escribano_appro}.
If we take $\Tilde{\varphi}$ sufficiently close to $\varphi$, we may assume that $\varphi(x)^\perp \cap \Tilde{\varphi}(x) = \{0\}$ for any $x \in M$.
Let $\pi: \epsilon_M^n=\xi \oplus \xi^\perp \rightarrow \xi$ be the orthogonal projection.
Let $\Tilde{\xi}$ be the definable $C^r$ vector bundle on $M$ corresponding to the map $\Tilde{\varphi}$.
The restriction $\pi|_{\Tilde{\xi}}$ of $\pi$ to $\Tilde{\xi}$ gives a definable $C^0$ isomorphism between the vector bundles $\xi$ and $\Tilde{\xi}$.
\end{proof}

\section{Equivalence of definable $C^r$ bilinear forms over vector bundles with bilinear forms over projective modules}\label{sec:equi_bil}

The purpose of this section is to demonstrate an equivalence of the category of definable $C^r$ bilinear spaces over a definable $C^r$ manifold $M$ with the category of bilinear spaces over $\cdfr(M)$.  
We first define definable $C^r$ bilinear spaces.

\begin{definition}[Definable $C^r$ bilinear spaces]
Let $r$ be a nonnegative integer.
Consider a definable $C^r$ manifold $M$ and a definable $C^r$ vector bundle $\xi=(E,p,M)$ over $M$.
A definable $C^r$ function $s: E \oplus E = \{(u,v) \in E \times E\;|\; p(u)=p(v)\} \rightarrow \mathbb R$ is called \textit{definable $C^r$ bilinear form} over the definable $C^r$ vector bundle $\xi$ if the restriction of $s$ to the set $\{(u,v) \in E \times E\;|\; p(u)=p(v)=x\}$ is a nondegenerate symmetric bilinear form for any $x \in M$.
We call the bilinear form $s$ \textit{positive definite} if the inequality $s(v,v)>0$ is satisfied for any $v \in E$ with $v \not=0$ as a vector in the fiber $p^{-1}(p(v))$.
We define a \textit{negative definite} bilinear form in the same way. 
A \textit{definable $C^r$ bilinear space} $(\xi,s)$ over $M$ is a pair of definable $C^r$ vector bundle $\xi$ over $M$ and a definable $C^r$ bilinear form $s$ defined over it.

Let $M'$ be a definable $C^r$ manifold.
Let $(\xi=(E,p,M),s)$ and $(\xi'=(E',p',M'),s')$ be two definable $C^r$ bilinear spaces over $M$ and $M'$, respectively.
Let $\varphi=(u,f)$ be definable $C^r$ morphism from the vector bundle $\xi$ to the vector bundle $\xi'$.
It means that $u:E \rightarrow E'$ and $f:M \rightarrow M'$ are definable $C^r$ maps with $p' \circ u = f \circ p$ such that $u|_{p^{-1}(x)}$ is a linear map for any $x \in M$.  
It is a \textit{definable $C^r$ morphism} from the bilinear space $(\xi,s)$ to the bilinear space $(\xi',s')$ if the equality $s(v,w)=s'(u(v),u(w))$ holds true for any $(v,w) \in E \oplus E$.
We define a \textit{definable $C^r$ $M$-morphism} from $(\xi,s)$ to $(\xi',s')$ in the same way when the base spaces $M$ and  $M'$ are identical.

Two definable $C^r$ bilinear space $(\xi_1=(E_1,p_1,M),s_1)$ and $(\xi_2,s_2)$ over a definable $C^r$ manifold are \textit{definably $C^r$ isometric} if there exists a definable $C^r$ $M$-isomorphism $(u, \operatorname{id}): \xi \rightarrow \xi'$ between definable $C^r$ vector bundles such that $s(v,w)=s'(u(v),u(w))$ for any $(v,w) \in E_1 \oplus E_1$.

Let $\epsilon_M^d$ be the trivial bundle of rank $d$ over a definable $C^r$ manifold $M$.
A definable $C^r$ bilinear space $(\epsilon_M^d,b)$ is a \textit{trivial bilinear space of type $(r_+,r_-)$} if $d=r_++r_-$ and $b((x,(v_1,\ldots,v_d)), (x,(w_1,\ldots, w_d))) = \sum_{i=1}^{r_+} v_iw_i - \sum_{i=1}^{r_-} v_{i+r_+}w_{i+r_+}$. 
The notation $r_+\left<1\right> \perp r_-\left<-1\right>$ denotes this bilinear form.
\end{definition}

\begin{prop}\label{prop:positive}
Let $r$ be a nonnegative integer.
Any definable $C^r$ vector bundle over a definable $C^r$ manifold has a positive definite definable $C^r$ bilinear form defined over it.
\end{prop}
\begin{proof}
Let $\xi$ be a definable $C^r$ bilinear space over a definable $C^r$ manifold $M$.
We may assume that $\xi$ be a subbundle of a trivial bundle $\epsilon_M^n$ by Lemma \ref{lem:equiv}.
There exists a positive definite bilinear form over the trivial bundle $\epsilon_M^n$.
Its restriction to $\xi$ is a desired positive definite bilinear form over $\xi$.
\end{proof}

The notation $\sgl(d,\mathbb R)$ denotes the set of all $d \times d$ symmetric invertible matrices.

\begin{prop}\label{prop:bilequiv}
Let $r$ be a nonnegative integer.
 Let $\xi$ be a definable $C^r$ vector bundle of rank $d$ over a definable $C^r$ manifold $M$ with  a definable $C^r$ atlas $(U_i,\phi_i:U_i \times \mathbb R^d \rightarrow p^{-1}(U_i))_{i=1}^q$.
 The definable $C^r$ map $g_{ij}: U_i \cap U_j \rightarrow \gl(d, \mathbb R)$ is a transition map defined in Remark \ref{rem:vector_bundle} for any $1 \leq i, j \leq q$.

There exists a one-to-one correspondence of definable $C^r$ bilinear forms over $\xi$ with families of definable $C^r$ maps  $\{s_i:U_i \rightarrow \sgl(d,\mathbb R)\}_{i=1}^q$ satisfying the equality
\begin{center}
${}^t\!g_{ji}(x)s_j(x)g_{ji}(x)=s_i(x)$
\end{center}
for any $x \in U_i \cap U_j$.
\end{prop}
\begin{proof}
When a bilinear form $s$ over $\xi=(E,p,M)$ is given, we define definable $C^r$ maps $s_i:U_i \rightarrow \sgl(d, \mathbb R)$ by  ${}^t\!vs_i(x)w = s(\phi_i(x,v),\phi_i(x,w))$ for any $x \in U_i$ and $v,w \in \mathbb R^d$.
These maps satisfy the requirement.

On the other hand, if a family of definable $C^r$ maps  $\{s_i:U_i \rightarrow \sgl(d,\mathbb R)\}_{i=1}^q$ is given.
Let $\pi_i:U_i \times \mathbb R^d \rightarrow \mathbb R^d$ be the projection.
Define $s:E \oplus E \rightarrow \mathbb R$ by $s(v,w)= {}^t\!\left(\pi_i(\phi_i^{-1}(v))\right) s_i(p(v)) \pi_i(\phi_i^{-1}(w))$ for any $v,w \in p^{-1}(U_i)$ with $p(v)=p(w)$.
It is easy to check that it is a well-defined definable $C^r$ bilinear form over $\xi$.
\end{proof} 

\begin{definition}[Orthogonal sum and Tensor product]
Let $(\xi_1=(E_1,p_1,M),s_1)$ and $(\xi_2=(E_2,p_2,M),s_2)$ be definable $C^r$ bilinear spaces over the same definable $C^r$ manifold $M$.
Their \textit{orthogonal sum} is the bilinear space $(\xi_1 \oplus \xi_2, s_1 \perp s_2)$ whose vector bundle is the Whitney sum of $\xi_1$ and $\xi_2$, and the bilinear form $s_1 \perp s_2$ is defined as follows:
\begin{center}
$(s_1 \perp s_2)(x_1 \oplus x_2, y_1 \oplus y_2) = s_1(x_1,y_1)+s_2(x_2,y_2)$.
\end{center}
The \textit{tensor product} of two bilinear spaces $(\xi_1,s_1)$ and $(\xi_2,s_2)$ is the bilinear space whose vector bundle is the tensor product of $\xi_1$ and $\xi_2$, and the bilinear form $s_1 \otimes s_2$ is defined as follows:
\begin{center}
$(s_1 \otimes s_2)(x_1 \otimes x_2, y_1 \otimes y_2) = s_1(x_1,y_1)s_2(x_2,y_2)$.
\end{center}
\end{definition}

\begin{prop}
Let $r$ be a nonnegative integer.
Let $(\xi_1,s_1)$ and $(\xi_2,s_2)$ be definable $C^r$ bilinear spaces over the same definable $C^r$ manifold $M$.
Their orthogonal sum $(\xi_1,s_1) \perp (\xi_2,s_2)$ and tensor product $(\xi_1,s_1) \otimes (\xi_2,s_2)$ are definable $C^r$ bilinear spaces over $M$.
\end{prop}
\begin{proof}
Obvious.
\end{proof}

\begin{definition}
Let $(\xi=(E,p,N),s)$ be a definable $C^r$ bilinear spaces over a definable $C^r$ manifold $N$.
Consider a definable $C^r$ map $f:M \rightarrow N$ between definable $C^r$ manifolds.
The \textit{definable $C^r$ bilinear space $f^*(\xi, s)$ induced by $f$} is a definable $C^r$ bilinear space whose vector bundle is $f^*\xi$ and whose definable $C^r$ bilinear form is $f^*s$ defined by
\begin{equation*}
f^*s((x,v),(x,w))= s(v,w)
\end{equation*}
for $x \in M$ and $v,w \in E$ with $p(v)=p(w)=f(x)$. 
\end{definition}

We next review the definition of bilinear spaces over a commutative ring.
The following definitions are found in \cite[Section 15.1]{BCR} and \cite{MH}.
\begin{definition}[Bilinear spaces over a commutative ring]
Let $R$ be a commutative ring.
Let $P$ be a finitely generated projective module over $R$ and $b:P \times P \rightarrow R$ be a symmetric bilinear form.
The bilinear form $b$ is nondegenerate if the linear map $h_b:P \rightarrow P^{\vee}$ from $P$ into its dual, induced by $b$, is an isomorphism.
A \textit{bilinear space} over $R$ is a pair $(P,b)$, where $P$ is a finitely generated projective $R$-module and $b:P \times P \rightarrow R$ is a nondegenerate symmetric bilinear form.
A \textit{morphism} from a bilinear space $(P,b)$ to another bilinear space $(P',b')$ is a homomorphism $\varphi:P \rightarrow P'$ between $R$-modules with $b'(\varphi(v),\varphi(w)) = b(v,w)$ for any $v,w \in P$. 
An \textit{isometry} of two bilinear spaces $(P,b)$ and $(P',b')$ over $R$ is an  isomorphism $\varphi: P \rightarrow P'$ between $R$-modules with $b'(\varphi(v),\varphi(w)) = b(v,w)$ for any $v,w \in P$.

Let $(P_1,b_1)$ and $(P_2,b_2)$ be bilinear spaces over the same commutative ring $R$.
Their \textit{orthogonal sum} is the bilinear space $(P_1 \oplus P_2, b_1 \perp b_2)$ whose projective module is the direct sum of $P_1$ and $P_2$, and the bilinear form $b_1 \perp b_2$ is defined as follows:
\begin{center}
$(b_1 \perp b_2)(x_1 \oplus x_2, y_1 \oplus y_2) = b_1(x_1,y_1)+b_2(x_2,y_2)$.
\end{center}
The \textit{tensor product} of $(P_1,b_1)$ and $(P_2,b_2)$ is the bilinear space whose projective module is the tensor product of $P_1$ and $P_2$, and the bilinear form $b_1 \otimes b_2$ is defined as follows:
\begin{center}
$(b_1 \otimes b_2)(x_1 \otimes x_2, y_1 \otimes y_2) = b_1(x_1,y_1)b_2(x_2,y_2)$.
\end{center}
It is obvious that $P_1 \oplus P_2$ and $P_1 \otimes P_2$ are finitely generated projective $R$-modules because an $R$-module is projective if and only if it is a direct summand of a free module \cite[Proposition A3.1]{Eisenbud}.
\end{definition}

\begin{notation}\label{not:bil}
Let $M$ be a definable $C^r$ manifold, where $r$ is a nonnegative integer. 
The notation $\bdfr(M)$ denotes the category whose objects are definable $C^r$ bilinear spaces over $M$ and whose arrows are definable $C^r$ $M$-morphisms between them.
Let $R$ be a commutative ring.
The notation $\bilsp(R)$ denotes the category whose objects are bilinear spaces over $R$ and whose arrows are morphisms between bilinear spaces over $R$.

For any definable $C^r$ bilinear space $(\xi,s)$ over $M$, we define a bilinear space $(P,b)=\funcbil(\xi,s)$ over the commutative ring $\cdfr(M)$ as follows:
The finitely generated projective module $P$ is $\funcvec(\xi)$ and the bilinear form $b: \funcvec(\xi) \times \funcvec(\xi) \rightarrow \cdfr(M)$ over it is given by
\begin{equation*}
b(\sigma,\sigma')(x) = s(\sigma(x),\sigma'(x))
\end{equation*}
for any $\sigma, \sigma' \in \funcvec(\xi)$ and $x \in M$.
We denote the bilinear form $b$ as $\funcbil(s)$ by abuse of notation.
A definable $C^r$ $M$-morphism between definable $C^r$ bilinear spaces over $M$ is simultaneously a definable $C^r$ $M$-morphisms between definable $C^r$ vector bundles.
The notation $\funcbil(u)$ denotes the homomorphism $\funcvec(u)$ between finitely generated projective $\cdfr(M)$-modules for any definable $C^r$ $M$-morphism $u$ between definable $C^r$ bilinear forms.
\end{notation}

We show that the above map $\funcbil$ is a covariant functor from the category $\bdfr(M)$ to the category $\bilsp(\cdfr(M))$.

\begin{prop}\label{prop:func1}
Let $M$ be a definable $C^r$ manifold, where $r$ is a nonnegative integer. 
The map $\funcbil$ is a covariant functor from the category $\bdfr(M)$ to the category $\bilsp(\cdfr(M))$.
\end{prop}
\begin{proof}
It is easy to check that the map $b$ defined in Notation \ref{not:bil} is a bilinear form over $\funcvec(\xi)$.
We show that $\funcbil(u)$ is a morphism between bilinear spaces for any definable $C^r$ $M$-morphism $u$ from a definable $C^r$ bilinear space $(\xi,s)$ over $M$ to a definable $C^r$ bilinear space $(\xi',s')$ over $M$.
Set $(P,b)=\funcbil(\xi,s)$ and $(P',b')=\funcbil(\xi',s')$.
We have only to show that $b'(\funcbil(u)(\sigma), \funcbil(u)(\sigma'))=b(\sigma,\sigma')$ for any $\sigma,\sigma' \in \funcvec(\xi)$.
Let $x \in M$ be fixed.
We have 
\begin{equation*}
b'(\funcbil(u)(\sigma), \funcbil(u)(\sigma'))(x) = s'(u \circ \sigma(x), u \circ \sigma'(x)) = s(\sigma(x),\sigma'(x)) = b(\sigma,\sigma')(x)\text{.}
\end{equation*}
We have finished the proof.
\end{proof}

Finally, we prove Main Theorem  \ref{thm:main2} introduced in Section \ref{sec:intro}.

\begin{thm}\label{thm:main2_0}
Let $M$ be a definable $C^r$ manifold, where $r$ is a nonnegative integer. 
Then, the functor $\funcbil$ is an equivalence of the category $\bdfr(M)$ with the category $\bilsp(\cdfr(M))$.
\end{thm}
\begin{proof}
In the same way as the proof of Theorem \ref{thm:main1_0}, we have to show that the functor $\funcbil$ is faithful and full, furthermore; for any bilinear space $(P,b)$ over the ring $\cdfr(M)$, there is a definable $C^r$ bilinear space $(\xi,s)$ such that $\funcbil(\xi,s)$ is isotropic to $(P,b)$ by \cite[Chapter IV, Section 4, Theorem 1]{Category}.
The functor $\funcbil$ is obviously faithful because the functor $\funcvec$ is faithful.

We next show that the functor $\funcbil$ is full.
Consider an arbitrary morphism $\varphi:\funcbil(\xi,s) \rightarrow \funcbil(\xi',s')$ of the category $\bilsp(\cdfr(M))$. 
Since the functor $\funcvec$ is full by Theorem \ref{thm:main1_0}, there exists a definable $C^r$ morphism $u:\xi \rightarrow \xi'$ with $\funcvec(u)=\varphi$.
Let $\xi=(E,p,M)$.
We want to show that this $u$ is also a morphism between $(\xi,s)$ and $(\xi',s')$ in the category $\bdfr(M)$.
We have only to show $s(v,w)=s'(u(v),u(w))$ for any $v,w \in E$ with $p(v)=p(w)$.
Set $x=p(v)=p(w)$.
There exist definable $C^r$ sections $\sigma_v, \sigma_w$ of $\xi$ with $\sigma_v(x)=v$ and $\sigma_w(x)=w$ by Lemma \ref{lem:sections}.  
The following calculation indicates that $u$ is a morphism between definable $C^r$ bilinear spaces $(\xi,s)$ and $(\xi',s')$.
\begin{align*}
s'(u(v),u(w)) &= s'(u(\sigma_v(x)), u(\sigma_w(x))) = (\funcbil(s')(\funcvec(u)(\sigma_v), \funcvec(u)(\sigma_w)))(x)\\
&=(\funcbil(s')(\varphi(\sigma_v),\varphi(\sigma_w)))(x) = (\funcbil(s)(\sigma_v,\sigma_w))(x)\\
&=s(\sigma_v(x),\sigma_w(x))=s(v,w)\text{.}
\end{align*}

We finally construct a definable $C^r$ bilinear space $(\xi,s)$ such that $\funcbil(\xi,s)$ is isotropic to $(P,b)$ for a given bilinear space $(P,b)$ over the ring $\cdfr(M)$.
We have already shown that $\funcvec$ gives an equivalence of the category $\vdfr(M)$ with the category $\proj(C_{\text{df}}^r(M))$ in Theorem \ref{thm:main1_0}.
There exists a definable $C^r$ vector bundle $\xi=(E,p,M)$ of rank $d$ such that $\funcvec(\xi)$ is isomorphic to the module $P$.
We have only to construct a definable $C^r$ bilinear form $s$ over $\xi$ such that $\funcbil(\xi,s)$ is isotropic to $(P,b)$.

Let $\psi:\funcvec(\xi) \rightarrow P$ be an isomorphism between $\cdfr(M)$-modules.
For any $v,w \in E$ with $p(v)=p(w)=x$, there exists definable $C^r$ sections $\sigma_v, \sigma_w \in \funcvec(\xi)$ with $\sigma_v(x) =v$ and $\sigma_w(x)=w$ by Lemma \ref{lem:sections}.
We set 
\begin{equation*}
s(u,v) = (b(\psi(\sigma_v),\psi(\sigma_w)))(x)\text{.}
\end{equation*}
We want to demonstrate that $s:E \oplus E \rightarrow \mathbb R$ is a definable $C^r$ bilinear form over $\xi$.
We first show that $s$ is a well-defined map.
Let $\sigma'_v$ and $\sigma'_w$ be definable $C^r$ sections of $\xi$ with $\sigma_v'(x)=v$ and $\sigma_w'(x)=w$.
The notation $\mathfrak{m}_x$ denotes the set of definable $C^r$ functions vanishing at the point $x \in M$.
It is a maximal ideal of $\cdfr(M)$.
Since $\sigma_v-\sigma_v' \in \mathfrak{m}_x\funcvec(\xi)$, we have $b(\psi(\sigma_v),\psi(\sigma_w))-b(\psi(\sigma_v'),\psi(\sigma_w)) = b(\psi(\sigma_v-\sigma_v'),\psi(\sigma_w)) \in \mathfrak{m}_x$.
We get $b(\psi(\sigma_v),\psi(\sigma_w))(x) = b(\psi(\sigma_v'),\psi(\sigma_w))(x)$.
We can show that $b(\psi(\sigma_v'),\psi(\sigma_w))(x) = b(\psi(\sigma_v'),\psi(\sigma_w'))(x)$ in the same way.
We have shown that
\begin{equation*}
b(\psi(\sigma_v),\psi(\sigma_w))(x) = b(\psi(\sigma_v'),\psi(\sigma_w'))(x)\text{.}
\end{equation*}
It means that $s$ is a well-defined map.

We next show that $s$ is a definable $C^r$ bilinear form over $\xi$.
Let $(U_i,\phi_i:U_i \times \mathbb R^d \rightarrow p^{-1}(U_i))_{i=1}^q$ be a definable $C^r$ atlas of $\xi$.
We can get definable $C^r$ sections $\sigma_1, \ldots, \sigma_m$ which generate $p^{-1}(x)$ for any $x \in M$ by Lemma \ref{lem:sections}.
Fix $1 \leq i \leq q$.
We may assume that $\sigma_1(x), \ldots \sigma_d(x)$ generate $p^{-1}(x)$ for any $x \in U_i$ in the same way as the proof of Theorem \ref{thm:main1_0}.
Let $\tau$ be the definable $C^r$ isomorphism from $U_i \times \mathbb R^d \times \mathbb R^d$ onto $p^{-1}(U_i) \oplus p^{-1}(U_i)$ defined by $\tau(x,(a_1,\ldots,a_d),(b_1,\ldots,b_d))=(\sum_{j=1}^d a_j\sigma_j(x), \sum_{j=1}^d b_j\sigma_j(x))$.
The composition $s|_{p^{-1}(U_i) \oplus p^{-1}(U_i)} \circ \tau$ is given by 
\begin{equation*}
s|_{p^{-1}(U_i) \oplus p^{-1}(U_i)} \circ \tau(x,(a_1,\ldots,a_d),(b_1,\ldots,b_d)) = \sum_{j,k=1}^d a_j  b_k  b(\psi(\sigma_j),\psi(\sigma_k))(x)\text{.}
\end{equation*}
It is a definable $C^r$ bilinear form over $U_i \times \mathbb R^d$.
Hence, the restriction of $s$ to $p^{-1}(U_i) \oplus p^{-1}(U_i)$ is a definable $C^r$ bilinear form over the restriction $\xi|_{U_i}$ of $\xi$ to $U_i$.
It means that $s$ itself is a definable $C^r$ bilinear form over $\xi$.

We finally show that $\funcbil(\xi,s)$ is isotropic to $(P,b)$.
We have $\funcbil(s)(\sigma_1,\sigma_2)(x)=s(\sigma_1(x),\sigma_2(x))=(b(\psi(\sigma_1),\psi(\sigma_2)))(x)$ for any $\sigma_1,\sigma_2 \in \funcvec(\xi)$ and $x \in M$.
Hence, we have $\funcbil(s)(\sigma_1,\sigma_2)=b(\psi(\sigma_1),\psi(\sigma_2))$.
We have shown that $\funcbil(\xi,s)$ is isotropic to $(P,b)$ via the isotropy $\psi$.
\end{proof}

\begin{prop}\label{funcbilin_plus}
Let $M$ be a definable $C^r$ manifold, where $r$ is a nonnegative integer. 
Then, we have $\funcbil((\xi_1,s_1) \perp (\xi_2,s_2))= \funcbil(\xi_1,s_1) \perp \funcbil(\xi_2,s_2)$ and $\funcbil((\xi_1,s_1) \otimes (\xi_2,s_2))= \funcbil(\xi_1,s_1) \otimes \funcbil(\xi_2,s_2)$ for any definable $C^r$ bilinear spaces $(\xi_1,s_1)$ and $(\xi_2,s_2)$ over $M$.
\end{prop}
\begin{proof}
We omit the proof.
\end{proof}

\section{Homotopy theorems for definable $C^r$ vector bundles and bilinear spaces}\label{sec:homotopy}
In this section, we show the homotopy theorems introduced in Section \ref{sec:intro}.
We first show the homotopy theorem for definable $C^r$ vector bundles.
It can be shown following the standard argument in \cite{H} using the following three lemmas:

\begin{lem}\label{lem:unit2}
Let $M$ be a definable $C^r$ manifold, where $r$ is a nonnegative integer. 
Let $U$ and $V$ be definable open subsets of $M$ with $\overline{U} \subset V$.
The maps $\pi_1$ and $\pi_2$ are the projections of $M \times [0,1]$ onto the first and the second components, respectively.
Then, there exists a definable $C^r$ map $\varphi: M \times [0,1] \rightarrow M \times [0,1]$ satisfying the following conditions: 
\begin{itemize}
\item The restriction of $\varphi$ to $V^{c} \times [0,1]$ is the identity map. 
\item The restriction of the composition $\pi_2 \circ \varphi$ to $U \times [0,1]$ is constantly one. 
\item The equality $\pi_1 \circ \varphi(x,t)=x$ is satisfied for any $x \in M$. 
\end{itemize}
\end{lem}
\begin{proof}
Since $\overline{U}$ and $V^{c}$ are  closed definable sets, there exists definable $C^r$ functions $f,g: M \rightarrow \mathbb R$ with $f^{-1}(0)=\overline{U}$ and $g^{-1}(0)=V^{c}$ by Lemma \ref{lem:zeroset}. 
The map $\varphi: M \times [0,1] \rightarrow M \times [0,1]$ defined by
\begin{equation*}
\varphi(x,t) = \left(x ,\frac{f^2}{f^2+g^2}t + \frac{g^2}{f^2+g^2}\right)
\end{equation*}
satisfies the requirements.
\end{proof}

The proofs of the following lemmas are almost the same as \cite[Section 3.4]{H}.

\begin{lem}\label{lem:vecbun1}
Let $U$ be a definable open subset of a definable $C^r$ manifold, and $W$ be a definable open subset of $U \times \mathbb R$ containing $U \times [0,1]$.
Let $\varphi_1,\varphi_2,\varphi_3: U \rightarrow [0,1]$ be definable $C^r$ functions with $\varphi_1 < \varphi_2 < \varphi_3$ on $U$.
We consider a definable $C^r$ vector bundle $\xi=(E,p,W)$ over $W$.
Define definable sets $B_1$ and $B_2$ as follows:
\begin{align*}
B_1 &= \{(x,t) \in U \times [0,1]\;|\; \varphi_1(x) \leq t \leq \varphi_2(x) \} \text{ and}\\ 
B_2 &= \{(x,t) \in U \times [0,1]\;|\; \varphi_2(x) \leq t \leq \varphi_3(x) \} \text{.}
\end{align*} 
For $i=1,2$, let $U_i$ be a definable open subset of $W$ such that $B_i \subset U_i$ and the restriction $\xi|_{U_i}$ of $\xi$ to $U_i$ is definably $C^r$ isomorphic to a trivial bundle over $U_i$.
Then, there exists a definable open subset $W'$ of $W$ such that $B_1 \cup B_2 \subset W'$ and the restriction $\xi|_{W'}$ of $\xi$ to $W'$ is definably $C^r$ isomorphic to a trivial bundle over $W'$.
\end{lem}
\begin{proof}
Set $W'=\{(x,t) \in U_1\;|\; t \leq \varphi_2(x)\} \cup \{(x,t) \in U_2\;|\;  t \geq \varphi_2(x)\}$.
We have only to show that $\xi|_{W'}$ is definably $C^0$ isomorphic to a trivial bundle by Corollary \ref{cor:bundle_iso}.
Let $\xi|_{U_i}=(E_i,p_i,U_i)$ and $u_i:U_i \times \mathbb R^d \rightarrow E_i$ be a definable $C^r$ isomorphism.
Set $B_3 = B_1 \cap B_2$ and let $v_i$ be the restriction of $u_i$ to $B_3 \times \mathbb R^d$.
Then, $h=v_2 \circ v_1^{-1}: B_3 \times \mathbb R^d \rightarrow B_3 \times \mathbb R^d$ is a definable $C^r$ isomorphism.
There exists a definable $C^r$ map $g:U \rightarrow \gl(d ,\mathbb R)$ with $h(x,v)=(x,g(x)v)$.
Define a definable continuous map $u:W' \times \mathbb R^d \rightarrow p^{-1}(W')$ by
\begin{equation*}
u((x,t),v)=\left\{\begin{array}{ll}
u_1((x,t),v) & \text{ if } t \leq \varphi_2(x) \text{ and }\\
u_2((x,t),g(x)v) & \text{ if } t > \varphi_2(x) \text{.}
\end{array}\right.
\end{equation*}
The map $u$ is a definable continuous isomorphism between the restriction $\xi|_{W'}$ of $\xi$ to $W'$ and a trivial vector bundle.
\end{proof}

\begin{lem}\label{lem:vecbun2}
Let $M$ be a definable $C^r$ manifold and $W$ be a definable open subset of $M \times \mathbb R$ containing $M \times [0,1]$.
We consider a definable $C^r$ vector bundle $\xi=(E,p,W)$ over $W$.
There exist a finite definable open covering $\{U_i\}_{i=1}^q$ of $M$ and a definable open subset $V_i$ of $U_i \times \mathbb R$ containing $U_i \times [0,1]$ such that the restriction $\xi|_{V_i}$ of $\xi$ to $V_i$ is definably $C^r$ isomorphic to a trivial bundle. 
\end{lem}
\begin{proof}
Let $(W_i,\psi_i)_{i=1}^p$ be a definable $C^r$ atlas of $\xi$.
By definition, the restriction $\xi|_{W_i}$ is definably $C^r$ isomorphic to a trivial bundle for any $1 \leq i \leq p$.
By Lemma \ref{lem:cover01}, there exists a finite definable open covering $\{U_i\}_{i=1}^q$ of $M$ and definable $C^r$ functions $0=\varphi_{i,0} < \cdots < \varphi_{i,k} < \cdots < \varphi_{i,r_i}=1$ on $U_i$ such that, for any $1 \leq i \leq q$ and  $1 \leq k \leq r_i$, the definable set
\begin{equation*}
B_{i,k}=\{(x,t) \in U_i \times [0,1] \;|\; \varphi_{i,k-1}(x) \leq t \leq \varphi_{i,k}(x)\}
\end{equation*}
is contained in $W_j$ for some $1 \leq j \leq p$.
In particular, for any $1 \leq i \leq q$ and  $1 \leq k \leq r_i$, there exists a definable open subset $U_{i,k}$ of $U \cap (U_i \times \mathbb R)$ such that $B_{i,k} \subset U_{i,k}$ and the restriction $\xi|_{U_{i,k}}$ is definably $C^r$ isomorphic to a trivial bundle.
Apply Lemma \ref{lem:vecbun1}, then, for any $1 \leq i \leq q$, there exists a definable open subset $V_i$ of $U_i \times \mathbb R$ such that it contains $U_i  \times [0,1]$ and $\xi|_{V_i}$ is definably $C^r$ isomorphic to a trivial bundle.
\end{proof}

We are now ready to show the homotopy theorem for definable $C^r$ vector bundles.

\begin{thm}[Homotopy theorem for definable $C^r$ vector bundles]\label{thm:homvec}
Consider a definable $C^r$ manifold $M$, where $r$ is a nonnegative integer. 
Let $U$ be a definable open subset of $M \times \mathbb R$ containing $M \times [0,1]$.
Define the definable $C^r$ map $r:U \rightarrow U$ by $r(x,t)=(x,1)$.
Let $\xi=(E,p,U)$ be a definable $C^r$ vector bundle over $U$.
Then, shrinking $U$ if necessary, two definable $C^r$ vector bundles $\xi$ and $r^*(\xi)$ are definably $C^r$ isomorphic.
\end{thm}
\begin{proof}
We construct a definable $C^r$ isomorphism $\Tilde{u}: \xi \rightarrow r^*(\xi)$.
There exist a finite definable open covering $\{U_i\}_{i=1}^q$ and a definable open subset $V_i$ of $U_i \times \mathbb R$ containing $U_i \times [0,1]$ such that the restriction $\xi|_{V_i}$ of $\xi$ to $V_i$ is definably $C^r$ isomorphic to a trivial bundle by Lemma \ref{lem:vecbun2}.
Let $h_i:V_i \times \mathbb R^d \rightarrow p^{-1}(V_i)$ be a definable $C^r$ isomorphism for any $1 \leq i \leq q$.
Shrinking $U$ if necessary, we may assume that $U=\bigcup_{i=1}^q V_i$.
There exist definable open sets $U_i'$ such that $\overline{U_i'} \subset U_i$ and $\{U_i'\}_{i=1}^q$ is again a finite definable open covering of $M$ by Lemma \ref{lem:covering}. 
There exists a definable $C^r$ map $r_i: M \times [0,1] \rightarrow M \times [0,1]$ such that 
the restriction of $r_i$ to $U_i^{c} \times [0,1]$ is the identity map, the restriction of the composition $\pi_2 \circ r_i$ to $U_i' \times [0,1]$ is constantly one, and $\pi_1 \circ r_i(x,t)=x$ for any $x \in M$ by Lemma \ref{lem:unit2}. 
Here, $\pi_1: M \times \mathbb R \rightarrow M$ and $\pi_2: M \times \mathbb R \rightarrow \mathbb R$ are the projections.
We may assume that $r_i$ is a definable $C^r$ map defined on $U$, shrinking $U$ if necessary.
We have $r=r_q \circ \cdots \circ r_2 \circ r_1$.

Set $s_0 = \operatorname{id}$ and $s_i = r_i \circ \cdots \circ r_2 \circ r_1$ for $i>0$.
Set $s_i^*\xi=(s_i^*E,p_i,U)$.
Consider the map $h_{i,j}: V_j \times \mathbb R^d \rightarrow s_i^*E|_{V_j}$ given by
\begin{equation*}
h_{i,j}((x,t),w)=((x,t),h_j(s_i(x,t),w))
\end{equation*}
for any $(x,t) \in V_j$ and $w \in \mathbb R^d$.
It is a definable $C^r$ diffeomorphism by the definition of the induced vector bundle.
Let $u_i:s_{i-1}^*\xi \rightarrow s_{i-1}^*\xi$ be the definable $C^r$ morphism such that it is defined by $u_i(h_{i-1,i}((x,t),v)) = h_{i-1,i}(r_i(x,t),v)$ for any $((x,t),v) \in V_i \times \mathbb R^d$  and $u_i$ is the identity map off the set $p_{i-1}^{-1}(V_i)$.
Consider the definable $C^r$ $U$-morphism $\Tilde{u_i}:s_{i-1}^*\xi \rightarrow s_i^*\xi$ given by $\Tilde{u_i}(v)=(p_{i-1}(v),u_i(v))$.
We show that $\Tilde{u_i}$ is a definable $C^r$ $U$-isomorphism.
It is clear that $\Tilde{u_i}$ is a bijective definable $C^r$ $U$-morphism.
We have only to show that $\Tilde{u_i}$ is a $C^r$ diffeomorphism.
Let $(x,t) \in U$ be fixed.
There exists $1 \leq j \leq q$ with $(x,t) \in V_j$.
We have the following commutative diagram:
\begin{equation*}
\begin{diagram}
\node{V_j \times \mathbb R^d} \arrow{s,r}{\operatorname{id}_{V_j} \times (u_i \circ h_{i-1,j})} \arrow{e,t}{h_{i-1,j}} \node{p_{i-1}^{-1}(V_j)} \arrow{s,r} {\Tilde{u_i}|_{p_{i-1}^{-1}(V_j)}}\\
\node{\{((x,t),v) \in V_j \times s_{i-1}^*E\;|\;r_i(x,t)=p_{i-1}(v)\}} \arrow{e,t} {\simeq}  \node{p_{i}^{-1}(V_j)} 
\end{diagram}
\end{equation*}
The definable $C^r$ map $\operatorname{id}_{V_j} \times (u_i \circ h_{i-1,j})$ is a $C^r$ diffeomorphism. Hence, $\Tilde{u_i}$ is also a $C^r$ diffeomorphism.

Set $\Tilde{u}=\Tilde{u_q} \circ \cdots \circ \Tilde{u_2} \circ \Tilde{u_1}$.
The $U$-morphism $\Tilde{u}$ is a definable $C^r$ $U$-isomorphism between $\xi$ and $r^*(\xi)$.
\end{proof}

The following corollary is Main Theorem \ref{thm:main30}.
\begin{cor}\label{thm:main30_0}
Consider a definable $C^r$ manifold $M$, where $r$ is a nonnegative integer. 
Let $U$ be a definable open subset of $M \times \mathbb R$ containing $M \times [0,1]$.
Let $\Xi$ be a definable $C^r$ vector bundle over $U$.
Then, two definable $C^r$ vector bundles $\Xi|_{M \times \{0\}}$ and $\Xi|_{M \times \{1\}}$ are definably $C^r$ isomorphic.
\end{cor}
\begin{proof}
Immediate from Theorem \ref{thm:homvec}.
\end{proof}

\begin{definition}
Let $f,g:M \rightarrow N$ be definable $C^r$ maps between definable $C^r$ manifolds.
We call them \textit{definably $C^r$ homotopic} if there exists a definable $C^r$ map $H:M \times [0,1] \rightarrow N$ with $H(x,0)=f(x)$ and $H(x,1)=g(x)$ for any $x \in M$.
\end{definition}

\begin{cor}\label{cor:vechomotpoic}
Consider definable $C^r$ manifolds $M$ and $N$, where $r$ is a nonnegative integer. 
Let $f,g:M \rightarrow N$ be definably $C^r$ homotopic definable $C^r$ maps.
Let $\xi$ be a definable $C^r$ vector bundle over $N$.
Then, the induced definable $C^r$ vector bundles $f^*\xi$ and $g^*\xi$ are definably $C^r$ isomorphic. 
\end{cor}
\begin{proof}
Immediate from Corollary \ref{thm:main30_0}.
\end{proof}

\begin{cor}\label{cor:vechomotpoic_euclid}
Any definable $C^r$ vector bundle over $\mathbb R^n$ is definably $C^r$ isomorphic to a trivial bundle. 
\end{cor}
\begin{proof}
Consider the definable $C^r$ function $H: \mathbb R^n \times [0,1] \rightarrow \mathbb R^n$ given by 
\begin{equation*}
H(x_1,\ldots,x_n,t)=(tx_1,\ldots, tx_n)\text{.}
\end{equation*}
It means that the identity map on $\mathbb R^n$ and the constant map given by $c(x_1,\ldots, x_n)=(0,\ldots,0)$ are definably $C^r$ homotopic.
Any definable $C^r$ vector bundle $\xi$ is definably $C^r$ isomorphic to $c^*\xi$, which is a trivial bundle, by Corollary \ref{cor:vechomotpoic}.
\end{proof}

We begin to prove the homotopy theorem for definable $C^r$ bilinear spaces.
Its counterpart for semialgebraic vector bundles is proved in \cite[Chapter 15]{BCR}.
We first show a theorem on Nash vector bundles in order to prove the homotopy theorem for definable $C^r$ bilinear spaces.
The proof is the same as that of \cite[Theorem 15.1.6]{BCR}, which is a theorem on semialgebraic $C^0$ vector bundles.
We give a proof for readers' convenience.
Nash vector bundles are defined in \cite[Section 12.7]{BCR}.
We use the real spectrum $\rspec(\nash(M))$ of the ring $\nash(M)$ of Nash functions on an affine Nash manifold $M$ in the proof of the following theorem.
Real spectrum is defined and its features are investigated in \cite[Chapter 7]{BCR}.

\begin{thm}\label{thm:nash_bundle}
Let $M$ be an affine Nash manifold and $\xi=(E,p,M)$ be a Nash vector bundle of rank $d$.
Let $(U_i,\phi_i:U_i \times \mathbb R^d \rightarrow p^{-1}(U_i))_{i=1, \ldots, q}$ be a Nash atlas of $\xi$.
Let $s$ be a Nash bilinear form over $\xi$.
The families of Nash maps  $\{s_i:U_i \rightarrow \sgl(d,\mathbb R)\}_{i=1, \ldots, q}$ are induced from the Nash bilinear form $s$ defined in the same way as Proposition \ref{prop:bilequiv}.
Then, there exist a finite semialgebraic open covering $\{V_j\}_{i=1}^r$ of $M$, Nash maps $g_j:V_j \rightarrow \gl(d,\mathbb R)$ and nonnegative integers $ 0 \leq r_{+,j} \leq d$ for any $1 \leq j \leq q$ such that $V_j$ is included in some $U_i$ and 
\begin{equation*}
{}^t\!g_j(x)s_i(x)g_j(x) = \left(\begin{array}{cc} I_{r_{+,j}} & O\\ O & -I_{r_{-,j}} \end{array}\right)
\end{equation*}
for any $x \in V_j$, where $r_{-,j} = d - r_{+,j}$ and $I_m$ is the $m \times m$ identity matrix.
\end{thm}
\begin{proof}
We may assume that $\xi$ is a trivial bundle over $M$ without loss of generality.
Consider an arbitrary prime cone $\alpha \in \rspec(\nash(M))$.
Let $k(\alpha)$ be the real closure of $\nash(M)/\operatorname{supp}(\alpha)$ under the ordering induced from the prime cone $\alpha$.
A nondegenerate symmetric bilinear form $s_{\alpha}:k(\alpha)^d \times k(\alpha)^d \rightarrow k(\alpha)$ is induced from the bilinear form $s$ by \cite[Theorem 8.4.4, Theorem 8.5.2]{BCR}.
We can find $h_{\alpha} \in \gl(d,k(\alpha))$ and $0 \leq r_{+,\alpha} \leq d$ with ${}^t\!h_\alpha s_\alpha h_\alpha = \left(\begin{array}{cc} I_{r_{+,\alpha}} & O\\ O & -I_{r_{-,\alpha}} \end{array}\right)$ by Sylvester's law of inertia, where $r_{-,\alpha}=d-r_{+,\alpha}$.
We identify $\rspec(\nash(M))$ with $\widetilde{M}$ by \cite[Proposition 8.8.1]{BCR}.
There exist an open semialgebraic subset $W_\alpha$ of $M$ and a Nash map $h: W_\alpha \rightarrow \gl(d,\mathbb R)$ with $\alpha \in \widetilde{W_{\alpha}}$ and $h(\alpha)=h_\alpha$ by \cite[Proposition 8.8.3]{BCR}.
The definitions of $\widetilde{M}$, $\widetilde{W_{\alpha}}$ and $h(\alpha)$ are found in \cite[Proposition 7.2.2, Theorem 7.2.3]{BCR}.
Let $w_i: W_\alpha \rightarrow \mathbb R^d$ be the $i$-th columns of $h$ for all $1 \leq i \leq d$.
Apply the Gram-Schmidt orthogonalization process, and we define $u_i$ and $v_i$ as follows:
\begin{align*}
v_1 &= |s_\alpha(w_1,w_1)|^{-1/2}w_1 \text{,}  \\
v_2 &= |s_\alpha(u_2,u_2)|^{-1/2}u_2,  \quad \text{ where }  u_2 = w_2-\dfrac{s_\alpha(w_2,v_1)}{s_\alpha(v_1,v_1)}v_1 \text{,}\\
&\cdots &\\
v_d &= |s_\alpha(u_d,u_d)|^{-1/2}u_d,  \quad \text{ where }  u_d = w_d-\sum_{i=1}^{d-1}\dfrac{s_\alpha(w_d,v_i)}{s_\alpha(v_i,v_i)}v_i \text{.}
\end{align*} 
Consider the semialgebraic open set
\begin{equation*}
 V_\alpha:=\{x \in W_\alpha\;|\; s_\alpha(w_1(x),w_1(x))\not=0, s_\alpha(u_2(x),u_2(x))\not=0 , \ldots, s_\alpha(u_d(x),u_d(x))\not=0\}\text{.}
 \end{equation*}
 The maps $u_i$ and $v_i$ are Nash maps on $V_\alpha$ for all $1 \leq i \leq d$.
 We have $s_\alpha(w_i(\alpha),w_j(\alpha))=\pm \delta_{ij}$ for all $1 \leq i,j \leq d$ by the definition of $w_i$, where $\delta_{ij}$ denotes the Kronecker delta.
We can inductively show that $s_\alpha(u_i(\alpha),u_i(\alpha)) = \pm 1$ in $k(\alpha)$.
Hence, the point $\alpha$ is contained in the set $\widetilde{V_\alpha}$.
Define $g: V_\alpha \rightarrow \gl(d,\mathbb R)$ so that the $i$-th row of $g(x)$ is $v_i(x)$ for any $1 \leq i \leq d$. 
We have ${}^t\!g(x) s(x) g(x) = \left(\begin{array}{cc} I_{r_{+,\alpha}} & O\\ O & -I_{r_{-,\alpha}} \end{array}\right)$ for any $x \in V_\alpha$.

The family $\{\widetilde{V_\alpha}\}_{\alpha \in \rspec(\nash(M))}$ is an open covering of $\rspec(\nash(M))$.
Since $\rspec(\nash(M))$ is compact by \cite[Corollary 7.1.13]{BCR}, a finite subfamily $\{\widetilde{V_j}\}_{j=1}^r$ of $\{\widetilde{V_\alpha}\}_{\alpha \in \rspec(\nash(M))}$ covers $\rspec(\nash(M))$.
The family $\{V_j\}_{j=1}^r$ covers $M$ by \cite[Proposition 7.2.2]{BCR}.
By the definition of $V_j$, there exists a Nash map $g_j:V_j \rightarrow \gl(d,\mathbb R)$ with
$
{}^t\!g_j(x)s(x)g_j(x) = \left(\begin{array}{cc} I_{r_{+,j}} & O\\ O & -I_{r_{-,j}} \end{array}\right)
$
for any $x \in V_j$ and some $0 \leq r_{+,j} \leq d$, where $r_{-,j} = d - r_{+,j}$.
\end{proof}

Theorem \ref{thm:nash_bundle} has several useful corollaries.

\begin{cor}\label{cor:nash_bundle1}
There exist a semialgebraic open covering $\{U_i\}_{i=1}^p$ of $\sgl(d,\mathbb R)$ and Nash maps $g_i:U_i \rightarrow \gl(d,\mathbb R)$ and $0 \leq r_{+,i} \leq d$ such that 
\begin{equation*}
{}^t\!g_i(s)sg_i(s) = \left(\begin{array}{cc} I_{r_{+,i}} & O\\ O & -I_{r_{-,i}} \end{array}\right)
\end{equation*}
for any $s \in U_i$ and $1 \leq i \leq p$, where $r_{-,i} = d - r_{+,i}$.
\end{cor}
\begin{proof}
Consider the trivial bundle $\epsilon_{\sgl(d,\mathbb R)}^d$ over $\sgl(d,\mathbb R)$ of rank $d$.
Let $\pi_1:\sgl(d,\mathbb R) \times \mathbb R^d \rightarrow \sgl(d,\mathbb R)$ and $\pi_2:\sgl(d,\mathbb R) \times \mathbb R^d \rightarrow \mathbb R^d$ be the natural projections.
Define a Nash bilinear form $b: \epsilon_{\sgl(d,\mathbb R)}^d \oplus \epsilon_{\sgl(d,\mathbb R)}^d \rightarrow \mathbb R$ by $ b(v,w)={}^t\!\pi_2(v)S\pi_2(w)$ for any $v,w \in \sgl(d,\mathbb R) \times \mathbb R^d$ with $S=\pi_1(v)=\pi_1(w)$.
The corollary is obtained by applying Theorem \ref{thm:nash_bundle} to $\epsilon_{\sgl(d,\mathbb R)}^d$ and $b$.
\end{proof}

\begin{cor}\label{cor:triv_bilsp1}
Consider a definable $C^r$ manifold $M$, where $r$ is a nonnegative integer. 
Let $(\xi,s)$ be a definable $C^r$ bilinear space over $M$ of rank $d$.
There exist a finite definable open covering $\{V_i\}_{i=1}^q$ of $M$ and $0 \leq r_{+,i} \leq d$ such that the restriction of the bilinear space $(\xi,s)$ to $V_i$ is definably $C^r$ isotropic to a trivial bilinear space of type $(r_{+,i},r_{-,i})$ for any $1 \leq i \leq q$. Here, $r_{+,i}=d-r_{-,i}$.
\end{cor}
\begin{proof}
We may assume that $\xi$ is a trivial bundle over $M$ without loss of generality.
Let $b: M \rightarrow \sgl(d,\mathbb R)$ be the definable $C^r$ map induced from the bilinear form $s$.
Let $\{U_i\}_{i=1}^p$ be a finite semialgebraic open covering of $\sgl(d,\mathbb R)$ given in Corollary \ref{cor:nash_bundle1}.
Set $V_i=b^{-1}(U_i)$.
The restriction of $\xi$ to $V_i$ is obviously definably $C^r$ isotropic to the trivial bilinear space of type $(r_{+,i},r_{-,i})$ for some $0 \leq r_{+,i} \leq d$ and $r_{-,i}=d-r_{+,i}$.
\end{proof}

\begin{lem}\label{bilsp_decomp_pre}
Consider a definable $C^r$ manifold $M$, where $r$ is a nonnegative integer. 
Let $U$ and $V$ be definable open subsets of $M$ with $M= U \cup V$.
Let $(\xi,s)$ be a definable $C^r$ bilinear space over $M$ of rank $d$.
We have a definable $C^r$ isometry $\varphi:(\epsilon_U^{r_++r_-}, r_+\left<1\right> \perp r_-\left<-1\right>) \rightarrow (\xi,s)|_{U}$, a positive definite definable bilinear space $(\nu_+,s|_{\nu_+})$ of rank $r_+$ over $V$ and a negative definite definable bilinear space $(\nu_-,s|_{\nu_-})$ of rank $r_-$ over $V$ with $\nu_+\oplus\nu_- = \xi|_{V}$.
Then, there exists a decomposition $\xi = \xi_+ \oplus \xi_-$ by definable $C^r$ vector subbundles $\xi_+$ and $\xi_-$ of $ \xi$ over $M$ such that the restrictions $s|_{\xi_+}$ and $s|_{\xi_-}$ are positive and negative definite, respectively. 
\end{lem}
\begin{proof}
Set $d=r_++r_-$.
Let $\sigma: U \cap V \rightarrow \grassman(d,r_+)$ be the definable $C^r$ map given by $\varphi^{-1}((\nu_+)_x) = \{x\} \times \sigma(x)$.
Here, $\grassman(d,r_+)$ denotes the Grassmanian of $r_+$-dimensional subspaces of an $d$-dimensional vector space.
Note that the restriction of the bilinear form $r_+\left<1\right> \perp r_-\left<-1\right>$ to the vector space $\sigma(x)$ is positive definite for any $x \in U \cap V$.
Let $M(r_-,r_+)$ be the set of all real-valued $r_- \times r_+$ matrices.
Set $\mathcal V=\{V \in \grassman(d,r_+)\;|\; V \cap (\{0\} \times \mathbb R^{r_-}) = \{0\}\}$.
Consider the biregular isomorphism $\psi: M(r_-,r_+) \rightarrow \mathcal V$ given by $\psi(\theta)=\{(v,\theta v) \in \mathbb R^d\;|\; v \in \mathbb R^{r_+}\}$ and set $\Omega = \{ \theta \in M(r_-,r_+)\;|\; \|\theta v\|_{r_-} < \| v\|_{r_+} \text{ for any } 0 \not= v \in \mathbb R^{r_+}\}$.
Remember that $\| \cdot \|_{r_+}$ and $\| \cdot \|_{r_-}$ denote the Euclidean norm in $\mathbb R^{r_+}$ and $\mathbb R^{r_-}$, respectively.
The semialgebraic set $\Omega$ is convex.
We can immediately show that $r_+\left<1\right> \perp r_-\left<-1\right>$ is positive definite on the vector space $\psi(\theta)$ if and only if $\theta \in \Omega$.
Let $\lambda$ and $\mu$ be a definable $C^r$ partition of unity on $M$ subordinate to the cover $M=U \cup V$ given by Lemma \ref{lem:unity}.
Define the subbundle $\xi_+$ of $\xi$ as follows:
\begin{equation*}
(\xi_+)_x = \left\{\begin{array}{ll}
\varphi(\{x\} \times \psi(\mu(x) \psi^{-1}(\sigma(x))) & \text{ if } x \in U \cap V \text{,}\\
\varphi(\{x\} \times (\mathbb R^{r_+} \times \{0\})) & \text{ if } x \in U \setminus U \cap V \text{ and }\\
(\nu_+)_x & \text{ if } x \in V \setminus U \cap V \text{.}
\end{array}\right.
\end{equation*}
Since $\theta$ is convex, we have $\mu(x) \psi^{-1}(\sigma(x)) \in \Omega$ for any $x \in U \cap V$; hence, the restriction $s|_{\xi_+}$ of $s$ to $\xi_+$ is positive definite.
We can construct $\xi_-$ in the same way.
\end{proof}

\begin{cor}\label{cor:bilsp_decomp}
Consider a definable $C^r$ manifold $M$, where $r$ is a nonnegative integer. 
Let $(\xi,s)$ be a definable $C^r$ bilinear space over $M$ of rank $d$.
Then, there exists a decomposition $\xi = \xi_+ \oplus \xi_-$ by definable $C^r$ vector subbundles $\xi_+$ and $\xi_-$ of $ \xi$ over $M$  such that the restrictions $s|_{\xi_+}$ and $s|_{\xi_-}$ of the bilinear form $s$ to $\xi_+$ and $\xi_-$ are positive and negative definite, respectively. 
The subbundles $\xi_+$ and $\xi_-$ are uniquely determined up to definable $C^r$ isomorphism.
\end{cor}
\begin{proof}
By Corollary \ref{cor:triv_bilsp1}, there exist a finite definable open covering $\{V_i\}_{i=1}^q$ of $M$ and $0 \leq r_i \leq d$ such that the restriction of the bilinear space $(\xi,s)$ to $V_i$ is definably $C^r$ isotropic to a trivial bilinear space of type $(r_i,s_i)$ for any $1 \leq i \leq q$. 
Here, $s_i=d-r_i$.
The existence of decomposition follows from Lemma \ref{bilsp_decomp_pre} and induction.

We next show the uniqueness.
Let $\xi=\xi_+'\oplus \xi_-'$ be another decomposition. 
Let $i:\xi_+' \rightarrow \xi$ be the natural inclusion and $\pi:\xi =  \xi_+ \oplus \xi_- \rightarrow \xi_+$ be the natural projection.
The composition $\pi \circ i: \xi_+' \rightarrow \xi_+$ is a definable $C^r$ isomorphism.
Hence, $\xi_+$ is unique up to definable $C^r$ isomorphism.
\end{proof}

We return to the proof of the homotopy theorem for definable $C^r$ bilinear spaces.
The proof is similar to that of Theorem \ref{thm:homvec}.

\begin{lem}\label{lem:bilinear1}
Let $U$ be a definable open set, and $W$ be a definable open subset of $U \times \mathbb R$ containing $U \times [0,1]$.
Let $\varphi_1,\varphi_2,\varphi_3: U \rightarrow [0,1]$ be definable $C^r$ functions with $\varphi_1 < \varphi_2 < \varphi_3$ on $U$.
We consider a definable $C^r$ bilinear space $(\xi=(E,p,W),s)$ over $W$.
Define definable sets $B_1$ and $B_2$ as follows:
\begin{align*}
B_1 &= \{(x,t) \in U \times [0,1]\;|\; \varphi_1(x) \leq t \leq \varphi_2(x) \} \text{ and}\\ 
B_2 &= \{(x,t) \in U \times [0,1]\;|\; \varphi_2(x) \leq t \leq \varphi_3(x) \} \text{.}
\end{align*} 
For $i=1,2$, let $U_i$ be a definable open subset of $W$ such that $B_i \subset U_i$ and the restriction $(\xi,s)|_{U_i}$ of $(\xi,s)$ to $U_i$ is definably $C^r$ isotropic to a trivial bilinear space over $U_i$ of type $(r_+,r_-)$.
Then, there exists a definable open subset $W'$ of $W$ such that $B_1 \cup B_2 \subset W'$ and the restriction $(\xi,s)|_{W'}$ of $(\xi,s)$ to $W'$ is definably $C^r$ isotropic to a trivial bilinear space over $W'$ of type $(r_+,r_-)$.
\end{lem}
\begin{proof}
We may assume that $U$ is connected without loss of generality.
Let $\xi|_{U_i}=(E_i,p_i,U_i)$, and let $u_i:U_i \times \mathbb R^d \rightarrow E_i$ be a definable $C^r$ isometry for $i=1,2$.
Define $W'$, $B_3$ and $u$ in the same way as Lemma \ref{lem:vecbun1}.
We showed that the map $u$ is a definable continuous isomorphism between the restriction $\xi|_{W'}$ of $\xi$ to $W'$ and a trivial bundle in Lemma \ref{lem:vecbun1}.
Set $J(r_+,r_-)=\left(\begin{array}{cc} I_{r_+} & O \\ O & -I_{r_-} \end{array}\right)$.
We have $s(u_i((x,t),v),u_i((x,t),w))= {}^t\!vJ(r_+,r_-) w$ for any $(x,t) \in U_i$ and $v,w \in \mathbb R^d$.
We want to show that the map $u$ is a definable $C^0$ isometry between $(\xi,s)|_{W'}$ and $(\epsilon_{W'}^d, r_+ \left< 1 \right> \perp r_- \left< -1 \right>)$, where $d = r_++r_-$.
We have only to show that 
\begin{equation*}
s(u((x,t),v),u((x,t),w))= {}^t\!v J(r_+,r_-) w
\end{equation*}
for any $(x,t) \in W'$ and $v,w \in \mathbb R^d$.
When $t \leq \varphi_2(x)$, it is obvious because $u = u_1$.
Since $u_1((x,t),v)=u_2((x,t),g(x)v)$ on $B_3$, we have $u_1((x,\varphi_2(x)),v)=u_2((x,\varphi_2(x)),g(x)v)$.
We get 
\begin{align*}
 {}^t\!v J(r_+,r_-) w &= s(u_1((x,\varphi_2(x)),v),u_1((x,\varphi_2(x)),w))\\
 &=s(u_2((x,\varphi_2(x)),g(x)v),u_2((x,\varphi_2(x)),g(x)w))\\
 &= {}^t\!v {}^t\!g(x)J(r_+,r_-) g(x)w\text{.}
 \end{align*} 
We have shown that $ {}^t\!v J(r_+,r_-) w = {}^t\!v {}^t\!g(x)J(r_+,r_-) g(x)w$ for any $x \in U$.
When $t > \varphi_2(x)$, we finally obtain 
\begin{align*}
s(u((x,t),v),u((x,t),w)) &= s(u_2((x,t),g(x)v),u_2((x,t),g(x)w))\\
& = {}^t\!v {}^t\!g(x)J(r_+,r_-) g(x)w={}^t\!v J(r_+,r_-) w\text{.}
\end{align*}

There exists a semialgebraic open subset $\mathcal U$ of $\sgl(d,\mathbb R)$ and a Nash map $\Tilde{g}: \mathcal U \rightarrow \gl(d,\mathbb R)$ such that $J(r_+,r_-) \in \mathcal U$ and ${}^t\!\Tilde{g}(S)S\Tilde{g}(S)=J(r_+,r_-)$ for any $S \in \mathcal U$ by Corollary \ref{cor:nash_bundle1}.
Take a definable $C^r$ isomorphism $\Tilde{u}:\xi|_{W'} \rightarrow \epsilon_{W'}^d$ sufficiently close to $u$ as in the proof of Corollary \ref{cor:bundle_iso}.
Define a definable $C^r$ bilinear form $\Tilde{s}$ over the trivial bundle $\epsilon_{W'}^d$ by $\Tilde{s}( ((x,t),v), ((x,t),w) ) = s(\Tilde{u}((x,t),v), \Tilde{u}((x,t),w) )$ for any $(x,t) \in W'$ and $v,w \in \mathbb R^d$.
Let $\Tilde{S}:W' \rightarrow \sgl(d,\mathbb R)$ be the definable $C^r$ map naturally induced from the bilinear form $\tilde{s}$.
Since $\Tilde{u}$ is sufficiently close to $u$, the nondegenerate symmetric matrix $\Tilde{S}(x,t)$ is sufficiently close to $J(r_+,r_-)$ for any $(x,t) \in W'$.
In particular, we may assume that $\Tilde{S}(x,t) \in \mathcal U$ for any $(x,t) \in W'$.

Define a definable $C^r$ isomorphism $\psi:W' \times \mathbb R^d \rightarrow W' \times \mathbb R^d$ by $\psi((x,t),v)=((x,t),\Tilde{g}(\Tilde{S}(x,t))v)$ for any $(x,t) \in W'$ and $v \in \mathbb R^d$.
The definable $C^r$ isomorphism $\Tilde{u} \circ \psi^{-1}$ is a definable $C^r$ isotropy between $(\xi,s)|_{W'}$ and a trivial bilinear space over $W'$ of type $(r_+,r_-)$.
\end{proof}

\begin{lem}\label{lem:bilinear2}
Let $M$ be a definable $C^r$ manifold and $W$ be a definable open subset of $M \times \mathbb R$ containing $M \times [0,1]$.
We consider a definable $C^r$ bilinear space $(\xi,s)$ over $W$.
Then, there exist a finite definable open covering $\{U_i\}_{i=1}^q$ of $M$ and a definable open subset $V_i$ of $U_i \times \mathbb R$ containing $U_i \times [0,1]$ such that the restriction $(\xi,s)|_{V_i}$ of $(\xi,s)$ to $V_i$ is definably $C^r$ isotropic to a trivial bilinear space. 
\end{lem}
\begin{proof}
We can prove the lemma in the same way as Lemma \ref{lem:vecbun2} using Corollary \ref{cor:triv_bilsp1} and Lemma \ref{lem:bilinear1} in place of Lemma \ref{lem:vecbun1}.
\end{proof}

We are now ready to show the homotopy theorem for definable $C^r$ bilinear spaces.

\begin{thm}[Homotopy theorem for definable $C^r$ bilinear spaces]\label{thm:hombil}
Consider a definable $C^r$ manifold $M$, where $r$ is a nonnegative integer. 
Let $U$ be a definable open subset of $M \times \mathbb R$ containing $M \times [0,1]$.
Define the definable $C^r$ map $r:U \rightarrow U$ by $r(x,t)=(x,1)$.
Let $\xi=(E,p,U)$ be a definable $C^r$ vector bundle over $U$ and $s$ be a definable $C^r$ bilinear form over $\xi$.
Then, shrinking $U$ if necessary, two definable $C^r$ bilinear spaces $(\xi,s)$ and $r^*(\xi,s)$ are definably $C^r$ isometric.
\end{thm}
\begin{proof}
We may assume that $M$ is connected without loss of generality.
There exist a finite definable open covering $\{U_i\}_{i=1}^q$ and a definable open subset $V_i$ of $U_i \times \mathbb R$ containing $U_i \times [0,1]$ such that the restriction $(\xi,s)|_{V_i}$ of $(\xi,s)$ to $V_i$ is definably $C^r$ isometric to a trivial bundle of type $(r_+,r_-)$ by Lemma \ref{lem:bilinear2}.
Let $h_i:V_i \times \mathbb R^d \rightarrow p^{-1}(V_i)$ be a definable $C^r$ isometry for any $1 \leq i \leq q$.
Shrinking $U$ if necessary, we may assume that $U=\bigcup_{i=1}^q V_i$.

We define $r_i$, $s_i$, $u_i$ and $\Tilde{u}_i$ in the same way as the proof of Theorem \ref{thm:homvec}.
Set $s_i^*\xi=(s_i^*E,p_i,U)$.
The notation $s_i^*s$ denotes the induced bilinear form over $s_i^*\xi$.
We have shown that $\Tilde{u}_i$ is a definable $C^r$ isomorphism between the definable $C^r$ vector bundles $s_{i-1}^*\xi$ and $s_i^*\xi$ in the proof of Theorem \ref{thm:homvec}.
We want to demonstrate that $\Tilde{u}_i$ is a definable $C^r$ isometry between the definable $C^r$ bilinear spaces $s_{i-1}^*(\xi,s)$ and $s_i^*(\xi,s)$.
For that purpose, we have only to show that $s_i^*s(\Tilde{u}_i(v),\Tilde{u}_i(w))=s_{i-1}^*(v,w)$ for any $v,w \in s_{i-1}^* E$ with $p_{i-1}(v)=p_{i-1}(w)$.
Identifying $s_i^*E$ with the space $\{((x,t),v) \in U \times s_{i-1}^*E\;|\; p_{i-1}(v)=r_i(x,t)\}$, we have $s_i^*s(((x,t),v),((x,t),w))=s_{i-1}^*s(v,w)$ for any $v,w \in s_{i-1}^*E$ with $p_{i-1}(v)=p_{i-1}(w)=r_i(x,t)$.
We easily get $s_i^*s(\Tilde{u}_i(v),\Tilde{u}_i(w))=s_i^*s((r_i(x,t),v),(r_i(x,t),w))=s_{i-1}^*(v,w)$.
We have shown that $\Tilde{u}_i$ is a definable $C^r$ isometry.

Set $\Tilde{u}=\Tilde{u_q} \circ \cdots \circ \Tilde{u_2} \circ \Tilde{u_1}$.
The $U$-morphism $\Tilde{u}$ is a definable $C^r$ $U$-isometry between $(\xi,s)$ and $r^*(\xi,s)$.
\end{proof}

The following corollary is Main Theorem \ref{thm:main3}.
\begin{cor}\label{thm:main3_0}
Consider a definable $C^r$ manifold $M$, where $r$ is a nonnegative integer. 
Let $U$ be a definable open subset of $M \times \mathbb R$ containing $M \times [0,1]$.
Let $(\Xi,B)$ be a definable $C^r$ bilinear space over $U$.
Then, two definable $C^r$ bilinear spaces $(\Xi,B)|_{M \times \{0\}}$ and $(\Xi,B)|_{M \times \{1\}}$ are definably $C^r$ isometric.
Here, the notation $(\Xi,B)|_{M \times \{0\}}$ denotes the restriction of the bilinear space $(\Xi,B)$ to $M \times \{0\}$.
\end{cor}
\begin{proof}
Immediate from Theorem \ref{thm:hombil}.
\end{proof}

\begin{cor}\label{cor:bilhomotpoic}
Consider definable $C^r$ manifolds $M$ and $N$, where $r$ is a nonnegative integer. 
Let $f,g:M \rightarrow N$ be definably $C^r$ homotopic definable $C^r$ maps.
Let $(\xi,s)$ be a definable $C^r$ bilinear space over $N$.
Then, the induced definable $C^r$ bilinear spaces $f^*(\xi,s)$ and $g^*(\xi,s)$ are definably $C^r$ isometric. 
\end{cor}
\begin{proof}
Immediate from Corollary \ref{thm:main3_0}.
\end{proof}

\begin{cor}\label{cor:bilpositive2}
Let $\xi$ be a definable $C^r$ vector bundle over a definable $C^r$ manifold, where $r$ is a nonnegative integer. 
Consider two positive definable $C^r$ bilinear forms $s$ and $s'$ over $\xi$.
Two definable $C^r$ bilinear spaces $(\xi,s)$ and $(\xi,s')$ are definably $C^r$ isometric. 
\end{cor}
\begin{proof}
Let $\xi=(E,p,M)$.
Consider the definable $C^r$ vector bundle $\eta=(E \times \mathbb R,q, M \times \mathbb R)$, where $q:E \times \mathbb R \rightarrow M \times \mathbb R$ is the definable $C^r$ map given by $q(v,t)=(p(v),t)$ for any $v \in E$ and $t \in \mathbb R$.
We next consider the definable $C^r$ bilinear form $\sigma$ over $\eta$ given by $\sigma(v,t)=(1-t)s(v)+ts'(v)$.
The bilinear form $\sigma$ is positive definite for any $0 \leq t \leq 1$.
There exists a definable open subset $U$ of $M \times \mathbb R$ containing $M \times [0,1]$ such that the restriction $\sigma|_{q^{-1}(U)}$ of $\sigma$ to  $q^{-1}(U)$ is positive definite.
Two definable $C^r$ bilinear space $(\xi,s)$ and $(\xi,s')$ are definably $C^r$ isometric by applying Corollary \ref{thm:main3_0} to the bilinear space $(\eta,\sigma)|_U$.
\end{proof}

We show counterparts of Corollary \ref{cor:bundle_iso} and Theorem \ref{thm:bundle_iso2} for definable $C^r$ bilinear spaces.
\begin{thm}\label{thm:bilinear_iso1}
Let $r$ be a nonnegative integer.
Consider two definable $C^r$ bilinear spaces over a definable $C^r$ manifold which are definably $C^0$ isometric. 
They are definably $C^r$ isometric.
\end{thm}
\begin{proof}
Let $(\xi,s)$ and $(\xi',s')$ be two definable $C^r$ bilinear spaces.
Let $\varphi: (\xi,s) \rightarrow (\xi',s')$ be a definable $C^0$ isometry between them.
We demonstrate that $(\xi,s)$ and $(\xi',s')$ are definably $C^r$ isometric.
We may assume that $(\xi,s)$ is the Whitney sum of a positive definite definable $C^r$ bilinear space $(\xi_+,s_+)$ and a negative definite definable $C^r$ bilinear space $(\xi_-,s_-)$
by Corollary \ref{cor:bilsp_decomp}.
We may also assume that $(\xi',s')$ is the Whitney sum of a positive definite definable $C^r$ bilinear space $(\xi'_+,s'_+)$ and a negative definite definable $C^r$ bilinear space $(\xi'_-,s'_-)$
in the same way.
Let $\iota: \xi_+ \rightarrow \xi$ be the natural inclusion and $\pi':\xi' \rightarrow \xi'_+$ be the natural projection.
The composition $\pi' \circ \varphi \circ \iota$ is a definable $C^0$ isometry between $(\xi_+,s_+)$ and $(\xi'_+,s'_+)$.
We can show that $(\xi_-,s_-)$ and $(\xi'_-,s'_-)$ are definably $C^0$ isometric in the same way.
Hence, we may assume that $(\xi,s)$ and $(\xi',s')$ are positive definite.

Since two vector bundles $\xi$ and $\xi'$ are definably $C^0$ isomorphic, they are definably $C^r$ isomorphic by  Corollary \ref{cor:bundle_iso}.
Let $\psi:\xi \rightarrow \xi'$ be definable $C^r$ isomorphism.
The positive definite definable $C^r$ bilinear form $s$ over $\xi$ naturally induces  a positive definite definable $C^r$ bilinear form $\Tilde{s}$ over $\xi'$ via the isomorphism $\psi$.
Two definable bilinear spaces $(\xi',s')$ and $(\xi',\Tilde{s})$ are definably $C^r$ isometric by Corollary \ref{cor:bilpositive2}.
We have finished the proof.
\end{proof}

\begin{thm}\label{thm:bilinear_iso2}
Let $r$ be a nonnegative integer.
Any definable $C^0$ bilinear spaces over a definable $C^r$ manifold is definably $C^0$ isometric to a definable $C^r$ bilinear space over the same manifold. 
\end{thm}
\begin{proof}
The definable $C^0$ bilinear space $(\xi, s)$ over a definable $C^r$ manifold is the Whitney sum of a positive definite definable $C^0$ bilinear space and a negative definite definable $C^0$ bilinear space by Corollary \ref{cor:bilsp_decomp}.
We may assume that the $C^0$ bilinear space is positive definite without loss of generality.
There exists a definable $C^r$ vector bundle $\Tilde{\xi}$ definably $C^0$ isomorphic to $\xi$ by Theorem \ref{thm:bundle_iso2}.
Let $\varphi:\Tilde{\xi} \rightarrow \xi$ be a definable $C^0$ isomorphism.
The positive definite definable $C^0$ bilinear form $s$ over $\xi$ naturally induces  a positive definite definable $C^0$ bilinear form $s'$ over $\Tilde{\xi}$ via the isomorphism $\varphi$.
Especially, two definable $C^0$ bilinear spaces $(\xi,s)$ and $(\Tilde{\xi},s')$ are definably $C^0$ isometric.

On the other hand, there exists a definable $C^r$ bilinear form $\Tilde{s}$ over $\Tilde{\xi}$ by Proposition \ref{prop:positive}.
Two definable bilinear spaces $(\Tilde{\xi},s')$ and $(\Tilde{\xi},\Tilde{s})$ are definably $C^0$ isometric by Corollary \ref{cor:bilpositive2}.
We have finished the proof.
\end{proof}

\section{Grothendieck rings and Witt rings are isomorphic}\label{sec:grothendieck}
The goal of this section is Main Theorem \ref{thm:main4}.
We first review the definitions of the Grothendieck ring and the Witt ring over a commutative ring.

\begin{definition}[Grothendieck ring over a commutative ring]
Let $R$ be a commutative ring. 
The notation $\projsec(R)$ denotes the set of all isomorphism class of finitely generated projective $R$-modules.
We define an equivalence relation $\sim$ on $\projsec(R) \times \projsec(R)$.
Let $(P_1,P_2), (P_1',P_2') \in \projsec(R) \times \projsec(R)$.
They are equivalent, namely, $(P_1,P_2) \sim (P_1',P_2')$ if there exist a nonnegative integer $r$ and an isomorphism of $R$-modules $P_1 \oplus P_2' \oplus R^r \simeq P_1' \oplus P_2 \oplus R^{r}$.
It is easy to check that the relation $\sim$ is an equivalence relation using the fact that a projective module is a direct summand of a free module.

The set $K_0(R)= \projsec(R) \times \projsec(R)/\sim$ is the \textit{Grothendieck ring} of $R$.
The notation $[(P_1,P_2)]$ denotes the equivalence class of $(P_1,P_2)$.
The sum of  $[(P_1,P_2)]$ and $[(P_1',P_2')]$ is given by $[(P_1 \oplus P_1',P_2 \oplus P_2')]$.
The inverse of $[(P_1,P_2)]$ is given by $[(P_2,P_1)]$ and denoted by $-[(P_1,P_2)]$.
The notation $[P]$ denotes the element $[(P,0)]$.
Any element of $K_0(R)$ is of the form $[P_1]-[P_2]$.
The product in $K_0(R)$ is defined using the tensor product of modules.
\end{definition}

For instance, the Grothendieck rings $K_0(\cdfr(\mathbb R^n))$ are isomorphic to $\mathbb Z$ by Corollary \ref{cor:vechomotpoic_euclid} for all nonnegative integers $n$ and $r$.

\begin{definition}[Witt ring over a commutative ring]
Let $R$ be a commutative ring such that $2$ is invertible in $R$.
The bilinear space $H(P)$ given by $(P \oplus P^{\vee},\beta)$ for some finitely generated projective module $P$ is called a \textit{hyperbolic space}, where $P^\vee$ is the dual of $P$ and $\beta$ is the bilinear form given by $\beta(x \oplus \varphi, y \oplus \psi)=\psi(x)+\varphi(y)$.

Two bilinear spaces $(P,b)$ and $(P',b')$ over the ring $R$ are \textit{Witt equivalent} if there are two finitely generated projective modules $Q$ and $Q'$ such that $(P,b) \perp H(Q)$ and  $(P',b') \perp H(Q')$ are isometric.
The notation $W(R)$ denotes the set of all Witt equivalence classes of bilinear spaces over $R$.
The sum of two elements $[(P,b)]$ and $[(P',b')]$ of $W(R)$ is defined as $[(P,b) \perp (P',b')]$.
The product of them are defined as $[(P,b) \otimes_A (P',b')]$.
They are well-defined by \cite[Theorem I.7.3]{MH}.
The additive inverse of $[(P,b)]$ is $[(P,-b)]$ by \cite[Theorem 15.1.4]{BCR}.
\end{definition}

The following theorem is Main Theorem \ref{thm:main4}.
The proof is similar to \cite[Theorem 15.1.2]{BCR}.
\begin{thm}\label{thm:main4_0}
Let $M$ be a definable $C^r$ manifold, where $r$ is a nonnegative integer. 
The Grothendieck ring $K_0(C_{\text{df}}^r(M))$ of the ring $C_{\text{df}}^r(M)$ is isomorphic to the Witt ring $W(C_{\text{df}}^r(M))$ of the same ring.
The Grothendieck rings $K_0(C_{\text{df}}^r(M))$ and $K_0(C_{\text{df}}^0(M))$ are also isomorphic.
\end{thm}
\begin{proof}
Set $A = \cdfr(M)$.
We first define a ring homomorphism $\Delta: K_0(A) \rightarrow W(A)$.
Since $K_0(A)$ is generated by $[P]$, where $P$ is a finitely generated projective $A$-module, we have only to define $\Delta([P])$.
There exists a definable $C^r$ vector bundle $\xi$ over $M$ such that $\funcvec(\xi)$ is isomorphic to $P$ by Theorem \ref{thm:main1_0}.
There exists a positive definite definable $C^r$ bilinear form $\sigma$ over $\xi$ by Proposition \ref{prop:positive}.
A bilinear form $b$ over $P$ is the bilinear form induced from the bilinear form $\sigma$ over $\xi$.
The equivalent class $[(P,b)]$ in $W(A)$ does not depend on the choice of a positive definite definable $C^r$ bilinear form $\sigma$ on $\xi$ by Corollary \ref{cor:bilpositive2}.
We set $\Delta([P])=[(P,b)]$.
When $[P]=[P']$ in $K_0(A)$, there exists a nonnegative integer $s$ and an isomorphism $\varphi:P' \oplus A^s \simeq P \oplus A^s$.
Let $\Tilde{b}$ be a positive definite bilinear form over $P \oplus A^s$ and $\varphi^*\Tilde{b}$ be a positive definable bilinear form over $P' \oplus A^s$ induced by $\Tilde{b}$ via $\varphi$.
Two bilinear space $(P \oplus A^s,b)$ and $(P' \oplus A^s,\varphi^*\Tilde{b})$ are isometric.
Let $b'$ be another positive definite bilinear form over $P'$.
We get $[(P',b')]=[(P',\varphi^*\Tilde{b}|_{P'})] = [(P' \oplus A^s, \varphi^*\Tilde{b})]- [(A^s,\varphi^*\Tilde{b}|_{A^s})] =
[(P \oplus A^s, \Tilde{b})]- [(A^s,\Tilde{b}|_{A^s})]
=[(P, \Tilde{b}|_{P})] = [(P,b)]$ by Corollary \ref{cor:bilpositive2}.
Therefore, the map $\Delta$ is well-defined.
Note that $b \perp b'$ and $b \otimes_A b'$ are positive definite when $b$ and $b'$ are positive definite.
We have shown that $\Delta$ is a ring homomorphism.

We next define a map $\nabla:W(A) \rightarrow K_0(A)$.
Let $(P,b)$ be a bilinear space over $A$.
There exists a definable $C^r$ bilinear space $(\xi,s)$ over $M$ such that $\funcbil(\xi,s)$ is isometric to $(P,b)$ by Theorem \ref{thm:main2_0}. 
There exists a unique decomposition $(\xi,s) = (\xi_+,s_+) \perp (\xi_-,s_-)$ up to isomorphism by Corollary \ref{cor:bilsp_decomp} such that $(\xi_+,s_+)$ and $(\xi_-,s_-)$ are positive and negative definite definable $C^r$ bilinear spaces over $M$, respectively.
We set $\nabla([(P,b)]) = [\funcvec(\xi_+)]- [\funcvec(\xi_-)]$.
It is obvious that $\nabla([(P,b) \perp (P',b')]) = \nabla([(P,b)]) + \nabla([(P',b')])$.
We show that $\nabla$ is well-defined.
Let $(P',b')$ be another bilinear space over $A$ with $[(P,b)]=[(P',b')]$ in $W(A)$.
There exist finitely generated projective $A$-modules $Q$ and $Q'$ such that  $(P,b) \perp H(Q)$ and $(P',b') \perp H(Q')$ are isometric.
Let $\Tilde{b}$ be a positive definite bilinear form over $Q$.
It exists by  Theorem \ref{thm:main1_0}, Proposition \ref{prop:positive} and Theorem \ref{thm:main2_0}.
The orthogonal sum $(Q, \Tilde{b}) \perp (Q, -\Tilde{b})$ is isomorphic to $H(Q)$ by the proof of \cite[Theorem 15.1.4]{BCR}.
We have $\nabla([(P,b)])=\nabla([(P,b)] )+ \nabla([(Q, \Tilde{b})]) + \nabla([(Q, -\Tilde{b})]) = \nabla([(P,b) \perp H(Q)]) = \nabla([(P',b') \perp H(Q')]) =\nabla([(P',b')])$.
It shows that $\nabla$ is well-defined.

We can easily show that $\nabla \circ \Delta$ is the identity map.
We next show that $\Delta \circ \nabla$ is the identity map.
Let $(P,b)$ be a bilinear space over $A$.
We define a definable $C^r$ bilinear spaces $(\xi,s)$, $(\xi_+,s_+)$ and $(\xi_-,s_-)$ over $M$ in the same way as above.
We have $\Delta \circ \nabla([(P,b)]) = \Delta([\funcvec(\xi_+)])-\Delta([\funcvec(\xi_-)])
=[\funcbil(\xi_+,s'_+)]-[\funcbil(\xi_-,s'_-)]$ for some positive definite definable $C^r$ bilinear forms $s'_+$ and $s'_-$ by Theorem \ref{thm:main2_0}.
We have $[\funcbil(\xi_+,s'_+)] = [\funcbil(\xi_+,s_+)]$ and $[\funcbil(\xi_-,s'_-)] = [\funcbil(\xi_-,-s_-)]$ by Corollary \ref{cor:bilpositive2}.
We get $\Delta \circ \nabla([(P,b)])=[\funcbil(\xi_+,s'_+)]-[\funcbil(\xi_-,s'_-)]=[\funcbil(\xi_+,s_+)]+[\funcbil(\xi_-,s_-)]=[\funcbil(\xi_+,s_+) \perp \funcbil(\xi_-,s_-)] = [\funcbil(\xi_+ \oplus \xi_-,s_+ \perp s_-)] = [\funcbil(\xi,s)]=[(P,b)]$.
We have shown that $\Delta \circ \nabla$ is the identity map.
The ring homomorphism $\Delta$ is a ring isomorphism between $K_0(A)$ and $W(A)$.

We finally demonstrate that the Grothendieck rings $K_0(C_{\text{df}}^r(M))$ and $K_0(C_{\text{df}}^0(M))$ are also isomorphic.
Since a definable $C^r$ vector bundle is simultaneously a definable $C^0$ vector bundle, we can consider a ring homomorphism $\Psi:K_0(\cdfr(M)) \rightarrow K_0(\cdfz(M))$ given by $\Psi([\funcvec(\xi)])=[\funcvec(\xi)]$ by Theorem \ref{thm:main1_0}. 
It is injective by Corollary \ref{cor:bundle_iso}.
It is surjective by Theorem \ref{thm:bundle_iso2}.
\end{proof}


\begin{thebibliography}{HD}

\normalsize
\baselineskip=17pt

\bibitem[BBK]{B1}
J. Bochnak, M. Buchner and W. Kucharz, 
{Vector bundles over real algebraic varieties},
K-theory \textbf{3} (1990), 271-298.
{Erratum},
K-theory \textbf{4} (1990), 103.

\bibitem[BCR]{BCR}
J. Bochnak, M. Coste and M. -F. Roy,
{Real algebraic geometry}.
Ergeb. Math. Grenzgeb.(3), Vol. 36.
Springer-Verlag: Berlin, 1998.

\bibitem[BK1]{B2}
J. Bochnak and W. Kucharz, 
{Algebraic approximation of mappings into spheres},
Michigan Math. J. \textbf{34} (1987), 119-125.

\bibitem[BK2]{B3}
J. Bochnak and W. Kucharz, 
{Realization of homotopy classes by algebraic mappings},
J. Reine Angew. Math. \textbf{377} (1987), 159-169.

\bibitem[BK3]{B4}
J. Bochnak and W. Kucharz,
{On real algebraic morphisms into even dimensional spheres},
Ann. of Math. \textbf{128} (1988), 415-433.

\bibitem[BK4]{B5}
J. Bochnak and W. Kucharz,
{Algebraic models of smooth manifolds},
Invent. Math. \textbf{97} (1989), 585-611.

\bibitem[BK5]{B6}
J. Bochnak and W. Kucharz,
{K-theory of real algebraic surfaces and threefolds},
Math. Proc. Camblidge Philos. Soc. \textbf{106} (1989), 471-480.

\bibitem[BK6]{B7}
J. Bochnak and W. Kucharz, (1990).
{On vector bundles and real algebraic morphisms},
in: Real analytic and algebraic geometry (Trento, 1988),
M. Galbiati and A. Tognoli (eds.), Lecture Note in Math. {1420}, Springer-Verlag: Berlin, 65-71.

\bibitem[BK7]{B8}
J. Bochnak and W. Kucharz, 
{Vector bundles on a product of real cubic curves},
K-theory \textbf{6} (1992), 487-497.

\bibitem[BK8]{B9}
J. Bochnak and W. Kucharz, 
{Algebraic cycles and approximation theorems in real algebraic geometry},
Trans. Amer. Math. Soc. \textbf{337} (1993), 463-472.

\bibitem[BK9]{B10}
J. Bochnak and W. Kucharz, 
{Elliptic curves and real algebraic morphisms},
J. Algebraic Geom. \textbf{2} (1993), 635-666.

\bibitem[CKP]{C}
M. Choi, T. Kawakami and D. H. Park, 
{Equivariant semialgebraic vector bundles},
Topology and its applications \textbf{123} (2002), 383-400.

\bibitem[Ei]{Eisenbud}
D. Eisnbud, 
{Commutative algebra with a view toward algebraic geometry},
Springer-Verlag: Berlin, 2004.

\bibitem[Es]{Escribano}
J. Escribano,
{Approximation theorems in o-minimal structures},
Illinois J. Math. \textbf{46} (2002), 111-128.

\bibitem[F]{Fisher}
A. Fisher,
{Smooth functions in o-minimal structures},
Advances in Math. \textbf{218} (2008), 496-514.

\bibitem[H]{H}
D. Husemoller,
{Fibre bundles}, 3rd ed.,
Springer-Verlag: New York, 1994.

\bibitem[Ka1]{K0}
T. Kawakami, 
{Every definable $C^r$ manifold is affine}, 
Bull. Korean Math. Soc. \textbf{42} (2005), 165-167.

\bibitem[Ka2]{K1}
T. Kawakami, 
{Definable $C^r$ fibre bundles and definable $C^rG$ vector bundles},
Commun. Korean Math. Soc. \textbf{23} (2008), 257-268.

\bibitem[Ku]{B11}
W. Kucharz,
{Vector bundles over real algebraic surfaces and threefolds},
Compositio Math. \textbf{60} (1986), 209-225.

\bibitem[Mac]{Category}
S. Mac Lane, 
{Categories for the working mathematician, 2nd eds.}, 
Springer-Verlag, New York, 1998.

\bibitem[Miln]{Milnor}
J. Milnor, 
{Introduction to algebraic K-theory}, 
Annals of Math. Studies, 72,
Princeton Univ. Press, Princeton, 1971.

\bibitem[MH]{MH}
J. Milnor and D. Husemoller, 
{Symmetric bilinear form}, 
Springer-Verlag, Berlin, 1973.

\bibitem[vdD]{vdD}
L. van den Dries, 
{Tame topology and o-minimal structures},
London Mathematical Society Lecture Note Series, Vol. 248.
Cambridge University Press, Cambridge, 1998.

\bibitem[vdDM]{vdDM}
L. van den Dries and C. Miller,
{Geometric categories and o-minimal structures},
Duke Math. J. \textbf{84} (1996), 497-540.

\bibitem[W]{Weibel}
C. Weibel,
{The K-book: an introduction to algebraic K-theory}.
Graduate studies in mathematics, 145,
American mathematical society, Providence, 2013.
\end{thebibliography}
\end{document}